\documentclass[11pt,a4paper,reqno]{amsart}
\usepackage[utf8]{inputenc}
\usepackage[T1]{fontenc}
\usepackage{bbm}
\usepackage{amsaddr}
\usepackage{amsmath}
\usepackage{amsthm}
\usepackage{amsfonts}
\usepackage{amssymb}
\usepackage{graphicx}
\usepackage{subfig}
\usepackage{xcolor}
\usepackage{mathtools}
\usepackage{mathrsfs}
\usepackage{bm}
\usepackage{enumitem}
\usepackage[normalem]{ulem}
\usepackage[pdftex,
colorlinks=true,
linkcolor=blue,
citecolor=blue
]{hyperref}

\newtheorem{thm}{Theorem}[section]
\newtheorem{prop}[thm]{Proposition}
\newtheorem{lem}[thm]{Lemma}
\newtheorem{cor}[thm]{Corollary}
\newtheorem{definition}[thm]{Definition}
\newtheorem{rem}{Remark}
\newtheorem{assumption}{Assumption}

\everymath{\displaystyle}
\providecommand{\abs}[1]{{\left| #1 \right|}}
\providecommand{\norme}[1]{{\left\lVert #1 \right\rVert}}
\providecommand{\normeinf}[1]{{\norme{#1}}_{\infty}}

\def\N{\mathbb{N}}
\def\R{\mathbb{R}}
\newcommand{\Ci}[2]{\mathscr{C}^{#1}{\left(#2\right)}}
\newcommand{\Linf}[1]{L^\infty{\left(#1\right)}}
\newcommand{\Lq}[2]{L^{#1}{\left(#2\right)}}
\newcommand{\normeLq}[3]{{{\norme{#3}}_{\Lq{#1}{#2}}}}

\newcommand{\umean}{e^{\int_0^t \overline{u}(\tau)\, d\tau} }
\newcommand{\umeaninverse}{e^{-\int_0^t \overline{u}(\tau)\, d\tau} }
\newcommand{\vdroit}{\mathsf{v}}
\newcommand{\ftest}{\mathscr{C}^1_{\rm c}(0,T)}

\makeatletter
 \def\@textbottom{\vskip \z@ \@plus 1pt}
 \let\@texttop\relax
\makeatother

\title[Evolutionary branching]{Evolutionary branching via replicator-mutator equations}

\author{Matthieu Alfaro {\tiny and} Mario Veruete}
\address{IMAG, Université de  Montpellier, CC051, 34095, Montpellier, France}
\email{matthieu.alfaro@umontpellier.fr, mario.veruete@umontpellier.fr}

\begin{document}
\begin{abstract}
We consider a class of non-local reaction-diffusion problems, referred to as {\it replicator-mutator} equations in evolutionary genetics. For a confining fitness function, we prove well-posedness and write the solution explicitly, via some underlying Schrödinger spectral elements (for which we provide new and non-standard estimates). As a consequence, the long time behaviour is determined by the principal eigenfunction or {\it ground state}. Based on this, we discuss (rigorously and via numerical explorations) the conditions on the fitness function and the mutation rate for  {\it evolutionary branching} to occur.
\end{abstract}

\maketitle

\section{Introduction}\label{s:intro}

In this paper we  first study the existence, uniqueness and long time behaviour of solutions $u=u(t,x)$, $t>0$, $x\in \R$, to the integro-differential Cauchy problem
 \begin{equation}\label{eq}
\left\{\begin{array}{ll}
\frac{\partial u}{\partial t} = \sigma^2\, \frac{\partial^2 u}{\partial x^2}+u \left( \mathcal{W}(x)-\int_\R \mathcal{W}(y)u(t,y)\, dy\right),\quad t>0,\, x\in \R, \vspace{5pt}\\
u(0,x)   = u_0(x),
\end{array}
\right.
\end{equation}
which serves as a model for the dynamics of adaptation, and where $\mathcal W$ is a confining fitness function (see below for details).  Next, we enquire on the possibility, depending on the function $\mathcal W$ and the parameter $\sigma>0$, for a solution to split from {\it uni-modal} to 
{\it multi-modal} shape, thus reproducing  {\it evolutionary branching}.

The above equation is referred to as a \emph{replicator-mutator} model. This type of model has found applications in different fields such as economics and biology \cite{PhysRevLett.80.2012}, \cite{AllenRosenbloom2012}.  In the field of evolutionary genetics, a free spatial version of equation \eqref{eq} was introduced by Tsimring, Levine and Kessler in \cite{PhysRevLett.76.4440}, where they propose a mean-field theory for the evolution of RNA virus population. Without mutations, and under the constraint of constant mass $\textstyle \int_\mathbb{R} u(t,x)\, dx=1$, the dynamics   is given by
\begin{equation}\label{nodiffusion}
\frac{\partial u}{\partial t} =u\left(\mathcal{W}(x)-\int_\mathbb{R}\mathcal{W}(y) u(t,y)\,dy\right),
\end{equation} 
with $\mathcal W(x)=x$ in \cite{PhysRevLett.76.4440}. In this context, $u(t,x)$ represents the density of a population (at time $t$ and per unit of phenotypic trait) on a one-dimensional phenotypic trait space. The function $\mathcal{W}(x)$ represents the \emph{fitness} of the phenotype $x$ and models the  individual reproductive success; thus the non-local term 
\begin{equation*} 
\overline{u}(t):= \int_\mathbb{R} \mathcal{W}(y) u(t,y)\, dy
\end{equation*}
stands for the mean fitness at time $t$. 

As a first step to take into account evolutionary phenomena,  mutations are modelled by the local diffusion operator $\sigma^2 \partial^2_x$, where $\sigma^2$ is the mutation rate, so that equation \eqref{nodiffusion} is transferred into \eqref{eq}. We refer to the recent paper \cite{Wak-Fun-Yok-17} for a rigorous derivation of
the replicator-mutator problem \eqref{eq} from individual based models.

Equation \eqref{eq} is supplemented with a non-negative and bounded, initial data $u_0(\cdot)\geq 0$ such that $\textstyle \int_\mathbb{R} u_0(x)\,dx=1$, so that, \emph{formally}, $\textstyle \int_\mathbb{R} u(t,x)\, dx=1$ for later times. Indeed, integrating formally \eqref{eq} over $x\in\mathbb{R}$, the total mass \begin{equation*}
m(t):=\int_\mathbb{R}u(t,x)\, dx
\end{equation*} solves the initial value problem
\[\frac{d}{dt}{m}(t)=(1-m(t))\overline{u}(t),\ m(0)=1.\]
Hence, by Gronwall's lemma,  $m(t)=1$, as long as $\overline{u}(t)$  is meaningful. 

\medskip

The case of linear fitness function, $\mathcal{W}(x)=x$, was the first introduced in \cite{PhysRevLett.76.4440}, but little was known concerning  existence and behaviours of solutions. Let us here mention the main result of Biktashev \cite{Biktashev2014}: for compactly supported initial data, solutions converge, as $t$ goes to infinity, to a Gaussian profile, where the convergence is understood in terms of the moments of $u(t, \cdot)$. In a recent paper \cite{AlfaroCarles14}, Alfaro and Carles proved that, thanks to a tricky change of unknown based on the Avron-Herbst formula (coming from quantum mechanics), equation \eqref{eq} can be reduced to the heat equation. This enables to compute the solution explicitly and describe contrasted behaviours depending on the tails of the initial datum: either the solution is global and tends, as $t$ tends to infinity, to a Gaussian profile which is centred around $x(t)\sim t^2$ (acceleration) and is flattening (extinction in infinite horizon), or the solution becomes extinct in finite time (or even immediately) thus contradicting the conservation of the mass, previously formally observed. 

For quadratic fitness functions, $\mathcal{W}(x) = \pm x^2$, it turns out that the equation can again be reduced to the heat equation \cite{AlfaroCarles2017}, up to an additional use of the generalized lens transform of the Schrödinger equation. In the case $\mathcal{W}(x)=x^2$, for any initial data, there is extinction at a finite time which is always bounded from above by $T^*=\tfrac{\pi}{4\sigma}$. Roughly speaking, both the right and left tails quickly enlarge, so that, in order to
conserve the mass, the central part is quickly decreasing. Then the non-local mean fitness term $\textstyle \int_\R y^2 u(t,y)\, dy$ becomes infinite very quickly and equation (1) becomes meaningless (extinction).
On the other hand, when $\mathcal{W}(x)=-x^2$, for any initial data, the solution is global and tends, as $t$ tends to infinity, to an universal stationary Gaussian profile.

The aforementioned cases $\mathcal W(x)=x$ and $\mathcal W(x)=x^2$ share the property of being unbounded from above, meaning that some phenotypes are infinitely well-adapted. This unlimited growth rate of
$u(t,x)$ in \eqref{eq} yields rich mathematical behaviours (acceleration, extinction) but is not admissible for biological applications. To deal
with such a problem, for the linear fitness case, some works
consider a ``cut-off version'' of \eqref{eq} at large $x$
\cite{PhysRevLett.76.4440}, \cite{RWC02}, \cite{SG10}, or provide a proper stochastic treatment for large phenotypic trait region \cite{RBW08}. 

On the other hand, $\mathcal W(x)=-x^{2}$ is referred to as a confining fitness function, typically preventing extinction phenomena.  However, it does not suffice to  take into account more realistic cases for which fitness functions are defined by a linear combination of two components (e.g. birth and death rates), each maximized by different optimal values of the underlying trait, a typical case being $\mathcal{W}(x)=x^2-x^4$.

Our main goal is thus to provide a rigorous treatment of the Cauchy problem \eqref{eq} when the fitness function $\mathcal W$ is confining. For a relatively large class of such fitness functions, we prove well-posedness, and show that the solution of \eqref{eq} converges to the principal eigenfunction (or ground state) of the underlying Schrödinger operator divided by its mass. This requires rather non-standard estimates on the eigenelements. Also, from a modelling perspective, this enables to reproduce 
 \emph{evolutionary branching},  consisting of the spontaneous splitting from uni-modal to multi-modal distribution of the trait.

 Such  splitting phenomena have long been discussed and analysed in different frameworks, see e.g. \cite{Lor-Mir-Per-11} via Hamilton-Jacobi technics, \cite{WI13} within finite populations, or \cite{MeleardMirrahimi2015} for a Lotka-Volterra system in a bounded domain. In a replicator-mutator context, let us notice that, while branching in \eqref{eq}  is mainly induced by the fitness function, it was recently obtained in  \cite{Gil-17} through different means. Precisely, the authors study  the case of linear fitness $\mathcal{W}(x)=x$ but non-local diffusion ${J*
 u-u}$ (mutation kernel), namely
\begin{equation*}
\partial _t u={J* u-u} +\,  u \left(x-\int _\R yu(t,y)\,dy\right).
\end{equation*}
Their approach \cite{Martin1541}, \cite{Gil-17} is based on {\it Cumulant Generating Functions} (CGF): it turns out that the CGF satisfies a first order non-local partial differential equation that can be explicitly solved, thus giving access to many informations such as  mean trait, variance, position of the leading edge. When a purely deleterious mutation kernel $J$ balances the infinite growth rate of $\mathcal W(x)=x$, they reveal some branching scenarios.
  
\medskip
 
The paper is organized as follows. In Section \ref{s:spectral} we present the underlying linear material. In Section \ref{s:main} we prove the well-possessedness of the Cauchy problem associated to \eqref{eq}. We also provide an explicit expression of the solution and studies its long time behaviour. In Section \ref{s:branching} we discuss, through rigorous details or numerical explorations, the conditions on the shape of the fitness function $\mathcal W$ and on the mutation parameter $\sigma>0$ for branching phenomenon to occur. Finally, we briefly conclude in Section \ref{s:discussion}.

\section{Some spectral properties}\label{s:spectral}

In this section, we present some linear material. We first quote some very classical results \cite{titchmarsh_eigenfunction_1946}, \cite{ReedSimonVol4}, \cite{Agmon}, \cite{Hel-Rob-82}, \cite{Hel-Sjo-85}, \cite{Hel-book-88}, \cite{Eremenko2008} for Schrödinger operators, and then prove less standard estimates on the eigenfunctions, which are crucial for later analysis. 

\subsection{Confining fitness functions and eigenvalues properties}\label{ss:eigenvalues}

Confining fitness functions tend to $-\infty$ at infinity. In quantum mechanics, this corresponds to potentials describing the evolution of quantum particles subject to an external field force that prevents them from escaping to infinity, that is, particles have a high probability of presence in a bounded spatial region. 

\begin{assumption}[Confining fitness function]\label{ass:confining} The fitness function $\mathcal W$  is continuous and confining, that is 
\[ 
\lim_{\abs{x}\to \infty} \mathcal{W}(x)=-\infty.
\]
\end{assumption}

\begin{prop}[Spectral basis]\label{prop:eigen}
Let  $\mathcal{W}$ satisfy Assumption \ref{ass:confining}. Then the operator 
\begin{equation}\label{def:H}
{\mathcal{H}}:=-\sigma^2 \frac{d^{2}}{dx^{2}}-\mathcal{W}(x)
\end{equation}
is essentially self-adjoint on $\mathscr{C}^{\infty}_{c} (\R)$, and has discrete spectrum:  there exists an orthonormal basis $\{\phi_k\}_{k\in \N}$ of $L^2(\R)$ consisting of eigenfunctions of ${\mathcal{H}}$
\[
{\mathcal{H}} \phi_k = \lambda_k \phi_k, \quad \Vert \phi _k\Vert _{L^{2}(\R)}=1,
\]
with corresponding eigenvalues
\[
\lambda_0<\lambda_1\leq \lambda_2\leq \cdots \leq \lambda_k\to +\infty,
\]
 of finite multiplicity.
\end{prop}

\begin{rem}
Proposition \ref{prop:eigen} is a classical result  \cite{ReedSimonVol4}, \cite[Chapter 3, Theorem 1.4 and Theorem 1.6]{Takhtajan} for Schrödinger operators with confining potential  $V_{\rm conf}(x)=-\mathcal{W}(x)$. In this paper, the minus sign is due to the biological interpretation of the fitness function $\mathcal{W}$.
\end{rem}

In the quantum mechanics terminology, $\phi_0$ is known as the \emph{ground state}, corresponding to the bound-state of minimal energy $\lambda_0$. In this paper we refer  to the couple $(\phi_0,\lambda_0)$ indistinctly as ground state/ground state energy or as principal eigenfunction/principal eigenvalue. 

The  principal eigenvalue $\lambda_0$ can be characterised by the  variational formulation%
\begin{equation}\label{VariationalPrinciple}
\lambda_0= \inf \left\{\mathcal{E}(u): u\in \mathscr{C}^{\infty}_{c} (\R), \normeLq{2}{\R}{u}=1 \right\},
\end{equation}
where $\mathcal{E}$ is the energy functional given by
\begin{equation*}
\mathcal{E}(u)=\sigma^2 \int_\R \abs{\frac{\partial u}{\partial x}}^2 \,dx + \int_\R -\mathcal{W}(x) \abs{u(x)}^2 \,dx.
\end{equation*}
Using concentrated test functions, the above formula enables to understand the behaviour of the principal eigenvalue $\lambda_0=\lambda_0(\sigma)$ as the mutation rate $\sigma$ tends to 0. The following will be used in Section \ref{s:branching} to prove some branching phenomena.
 
\begin{prop}[Asymptotics for $\lambda_0(\sigma)$ as $\sigma \to 0$]\label{prop:sigma-pt}
Let $\mathcal{W}$ satisfy Assumption \ref{ass:confining}. Assume that $\mathcal W$ reaches a global maximum $M$ at $x=\alpha$. Then $\lim_{\sigma\to 0^+}\lambda_0(\sigma) =-M$.
\end{prop}

\begin{proof} For the convenience of the reader, we give the proof of this standard fact. Let $p$ be a smooth, non-negative, and compactly supported in $[-1,1]$ function with $\Vert p\Vert _{L^2(\R)}=1$. We define the test function 
\[
p_\sigma(x):=\frac{1}{\sigma ^{1/4}}p\left(\frac{x-\alpha}{\sqrt \sigma}\right).
\] 
From the variational formula \eqref{VariationalPrinciple}, we have
\[-M\leq \lambda_0(\sigma)\leq \sigma^2 \int_\R \abs{{\partial_x p_\sigma}(x)}^2 \,dx + \int_\R -\mathcal{W}(x) \abs{p_\sigma (x)}^2 \,dx.\]
The first integral in the right hand side is given by
\[\sigma^2 \int_\R \abs{\partial _x p_\sigma (x)}^2  \,dx = \sigma  \normeLq{2}{\R}{p'}\to 0,\]
as $\sigma \to 0^{+}$. The second integral gives 
\[\int_\R -\mathcal{W}(x) \abs{p_\sigma(x)}^2 \,dx= \int_\R -\mathcal{W}(\alpha+\sqrt \sigma y) p(y)^2 \, dy\]
which,  by the $L^1$-dominated convergence theorem tends to $-M\normeLq{2}{\R}{p}=-M$  as $\sigma \to 0^{+}$.
\end{proof}

In the subsequent sections, we will quote results on the spectral properties of Schrödinger operators, in particular an asymptotics for the eigenvalues $\lambda _k$ as $k\to +\infty$. As far as we know, the available results require to assume  that the fitness $\mathcal{W}$ is polynomial.

\begin{assumption}[Polynomial confining fitness function]\label{ass:polynome} The fitness function $\mathcal{W}$ is a real polynomial of degree $2s$:
\[
\mathcal{W}(x)=-x^{2s}+\sum _{k=0}^{2s-1} w_k x^{k},
\]
for some integer $s \geq 1$ and some real numbers $w_k$, $0 \leq k \leq 2s-1$.
\end{assumption}

Under Assumption \ref{ass:polynome}, elliptic regularity theory insures that the eigenfunctions are infinitely differentiable. Furthermore, all the derivatives of each eigenfunction are square-integrable \cite{Gagelman2012}. Notice that it is also known that all eigenfunctions actually belong to the Schwartz space $\mathcal{S} (\R)$.

\begin{prop}[Asymptotics for eigenvalues]\label{prop:asymptotic}
Let  $\mathcal{W}$ satisfy Assumption \ref{ass:polynome}. Then all eigenvalues of $\mathcal H$ are simple and
\begin{equation}\label{asymptoticEigenvalues}
\lambda _k \sim {C_{s,\sigma}\, k}^{\frac{2s}{s+1}} \quad \text{ as } k\to +\infty,
\end{equation}
where $C_{s,\sigma}:= \left(\frac{\sigma \sqrt{\pi}\, \Gamma(\frac{3}{2}+\frac{1}{2s}) }{\Gamma(1+\frac{1}{2s})} \right)^{\frac{2s}{s+1}}
$, with $\Gamma (z)= \int\limits_{\R^+}  t^{z-1}\,e^{-t}\,\mathrm{d}t$ being the gamma function.
\end{prop}

We refer to \cite{titchmarsh_eigenfunction_1946}, \cite{Eremenko2008} and the references therein for more details on the above asymptotic formula. Furthermore, in the case of a symmetric fitness  $\mathcal{W}(-x)=\mathcal{W}(x)$, the simplicity of eigenvalues enforce all eigenfunctions to be even or odd. In particular the principal eigenfunction $\phi_0$ (\emph{ground state}) is even since it is known to have constant sign.

\begin{rem} Assume that $\mathcal W$ is such that $P-c\leq \mathcal W\leq P+c$ for some polynomial $P$ as in Assumption \ref{ass:polynome} and some constant $c>0$. From Courant-Fisher's theorem, that is the variational characterization of the eigenvalues,  we deduce that $\lambda_k^\prime-c\leq \lambda_k\leq \lambda^\prime_k+c$,
where $\lambda_k^\prime$ are the eigenvalues of the Hamiltonian with potential $P$. Hence, $\lambda_k$ share with $\lambda_k^\prime$  the asymptotics \eqref{asymptoticEigenvalues}, which is the  keystone for deriving the estimates  on eigenfunctions in subsection \ref{ss:eigenfunctions}, and thereafter our main results in Section \ref{s:main}. Hence, our results apply to such fitness functions, covering in particular the case of the  so-called pseudo-polynomials (i.e. smooth functions which coincide, outside of a compact region, with a polynomial $P$ as in Assumption \ref{ass:polynome}), which are relevant for numerical computations. 
\end{rem}

\subsection{\texorpdfstring{$L^1$}{L1}, \texorpdfstring{$L^\infty$}{Linf} and weighted  \texorpdfstring{$L^1$}{L1} norms of the eigenfunctions}\label{ss:eigenfunctions}

In the study of spectral properties of Schrödinger operators, efforts tend to concentrate around asymptotic estimates of  eigenvalues or on the regularity and decay of eigenfunctions  \cite{titchmarsh_eigenfunction_1946}, \cite{Agmon},  \cite{PhysRevA.43.3241}, \cite{Eremenko2008bis}. Much less attention has been given to estimate the $L^1$ and $L^\infty$ norms  of eigenfunctions. One reason is that the natural framework for eigenfunctions of the Hamiltonian ${\mathcal{H}}$, defined in \eqref{def:H}, is $L^2(\R)$. On the other hand, the biological nature of problem \eqref{eq} suggests  $L^1(\R)$ and $L^\infty(\R)$  as  natural spaces for the solution $u(t,x)$. We therefore provide in this subsection rather non-standard estimates on the eigenfunctions.

\medskip

We define \[m_k:=\int_\R \phi_k(x)\, dx,\] the mass of the $k$-th eigenfunction $\phi_k$ of the  Hamiltonian $\mathcal H$. In the sequel, by 
\[
A_k\lesssim B_k
\]
we mean that there is $c>0$ such that, for all $k\geq 1$, $A_k\leq cB_k$.

\begin{prop}[$L^1$ norm of  eigenfunctions]\label{prop:mass-estimate}
Let $\mathcal W$ satisfy  Assumption \ref{ass:polynome}. Then we have
\begin{equation}\label{masse}
\vert m_k \vert \leq {\norme{\phi_k}}_{L^1(\R)} \lesssim  k^{\tfrac{1}{2(s+1)}}. 
\end{equation}
\end{prop}

Before proving the above proposition, we need the following lemma which is of independent interest.

\begin{lem}\label{lem:normes}
Let $d\in \mathbb{N}^*$ and $N\in(d,+\infty)$ be given. Then there is a constant $C=C(d,N)>0$ such that, for all $f:\mathbb{R}^d\to \mathbb{R}$, 
\begin{equation}\label{lemmeRemi}
\norme{f}_{L^1(\R^d)} \leq C\ {{\norme{f}}}^{1-\delta}_{L^2(\R^d)}\; {\textstyle \norme{x^{N/2} f}}^{\delta}_{L^2(\R^d)}, \quad \delta:=d/ N.
\end{equation}
\end{lem}

\begin{proof}
Let   $\mathcal{B}_R$ denote the open $d$-dimensional ball of radius $R>0$ and center  $0_{\R^d}$. We write
\[
\norme{f}_{L^1(\R^d)}  =\int\limits_{\mathcal{B}_R} \abs{f(x)}\, dx+\int\limits_{\R^d\setminus \mathcal{B}_R} \frac{1}{\abs{x}^{N/2}} \abs{x}^{N/2} \abs{f(x)}\, dx  =:\mathcal{I}_1+\mathcal{I}_2.
\]
By the Cauchy-Schwarz inequality we have
\[
\mathcal{I}_1 \leq \norme{f}_{L^2(\mathcal{B}_R)}\norme{1}_{L^2(\mathcal{B}_R)}= {\left[\frac{\pi^{d/2} R^d}{\Gamma\left(1+\frac{d}{2}\right)} \right]}^{1/2}\norme{f}_{L^2(\R^d)}=:C_1 R^{d/2}\norme{f}_{L^2(\R^d)},
\]
and 
\begin{eqnarray*}
\mathcal{I}_2 & \leq & {\Big(\int_{\R^d\setminus \mathcal{B}_R}\frac{1}{\abs{x}^N} \,dx\Big)}^{1/2} {\textstyle \norme{x^{N/2} f}}_{L^2(\R^d)}\nonumber \\
\nonumber \\
& \leq & C \Big(\int_R^{+\infty} \frac{1}{r^N}\, r^{d-1}\, dr \Big)^{1/2} {\textstyle \norme{x^{N/2} f}}_{L^2(\R^d)} \nonumber\\
& \leq  & C_2 R^{(d-N)/2} {\textstyle \norme{x^{N/2} f}}_{L^2(\R^d)},
\end{eqnarray*}
for some $C=C(d)>0$ and $C_2=C_2(d,N)>0$.  Summarizing, 
\begin{equation}\label{summarize}
\norme{f}_{L^1(\R^d)} \leq  C_1 R^{d/2}\norme{f}_{L^2(\R^d)}+ C_2 R^{(d-N)/2} {\textstyle \norme{x^{N/2} f}}_{L^2(\R^d)}.
\end{equation}
Now, we select
\[
R=\left( \frac{C_2 (N-d)}{C_1 d}\frac{\norme{x^{N/2}f}_{L^2(\R^d)}}{\norme{f}_{L^2(\R^d)}} \right)^{2/N}
\]
which minimizes the right hand side of  \eqref{summarize} and yields \eqref{lemmeRemi}.
\end{proof}

\begin{rem}
The correct power $\delta$ in \eqref{lemmeRemi} can be retrieved by a standard homogeneity argument. Indeed, defining $f_\lambda (x):=f(\lambda x)$ for $\lambda >0$, we get
\[
\normeLq{1}{\R^d}{f_\lambda}=\frac{1}{\lambda ^d}  \normeLq{1}{\R^d}{f},
\]
\[
\normeLq{2}{\R^d}{f_\lambda}=  \frac{1}{\lambda^{d/2}} \normeLq{2}{\R^d}{f},
\]
\[
\normeLq{2}{\R^d}{x^{N/2}f_\lambda}=  \frac{1}{\displaystyle \lambda^{(N+d)/2}} \normeLq{2}{\R^d}{x^{N/2}f},
\]
so that
\[
\frac{1}{\lambda^d} \lesssim \left(\frac{1}{\lambda^{d/2}}\right)^{1-\delta} \left(\frac{1}{\displaystyle \lambda^{(N+d)/2}}\right)^\delta.
\]
Powers of  $\lambda$ in both sides must coincide, which enforces $\delta =d/N$.
\end{rem}

We can now estimate the mass of the eigenfunctions.

\begin{proof}[Proof of Proposition \ref{prop:mass-estimate}] Up to subtracting a constant to $\mathcal W$, we can assume without loss of generality that $\mathcal W<0$. Multiplying by $\phi_k$ the eigenvalue equation \[-\sigma^2 \phi_k''-\mathcal{W}\phi_k= \lambda_k \phi_k\] and integrating over $x\in\R$, we get \begin{equation*}
\int_\R -\sigma^2\phi_k''(x) \phi_k(x) \, dx + \int_{\R} -\mathcal{W}(x) \phi_k(x)^2\, dx = \lambda_k \int_\R \phi_k^2(x) \, dx.
\end{equation*} Integrating by parts and recalling that  eigenfunctions $\phi_k$ are normalized in $L^2(\R)$,  we obtain
\[
\sigma^2 \int_\R \phi_k'(x)^2 \,dx+ \int_\R -\mathcal{W}(x) \phi_k(x)^2\, dx = \lambda_k,
\]
so that
\[
\int_\R -\mathcal{W}(x) \phi_k(x)^2\, dx\leq \lambda_k.
\]
Next, it follows from Assumption \ref{ass:polynome} (and $\mathcal{W}<0$) that there is $\gamma >0$ such that $-\mathcal{W}(x)\geq \gamma x^{2s}$ for all $x\in \R$, and thus
\[
 {\norme{x^s \phi_k}}^2_{L^2(\R)}  \leq \frac{\lambda_k}{\gamma}.
\]
Now, by Lemma \ref{lem:normes}, we have
\begin{equation*}
{ \norme{\phi_k}}_{L^1(\R)} \lesssim   {\norme{x^{s}\phi_k}}_{L^2(\R)}^{1/(2s)}\lesssim \lambda_k^{1/4s},
\end{equation*}
which, combined with \eqref{asymptoticEigenvalues}, implies \eqref{masse}. The proposition is proved.
\end{proof}

\begin{prop}[$L^\infty$ norm of eigenfunctions]\label{prop:nomreinfiniefonctionspropres} Let $\mathcal W$ satisfy  Assumption \ref{ass:polynome}. Then we have
\begin{equation}\label{normeinfinie}
{\norme{\phi_k }}_{L^\infty(\R)} \lesssim k^{\tfrac{s}{2(s+1)}}.
\end{equation}
\end{prop}

\begin{proof} Since $-\frac 12 \phi_k^2(x)=\int_x^{+\infty}\phi_k(s)\phi_k'(s)ds$, we have
$$
{\norme{\phi_k }}_{L^\infty(\R)}^{2}\lesssim {\norme{\phi_k }}_{L^2(\R)}{\norme{\phi_k ' }}_{L^2(\R)}\lesssim \lambda _k^{1/2},
$$
and the conclusion follows from \eqref{asymptoticEigenvalues}.
\end{proof}

\begin{prop}[Weighted $L^1$ norm of eigenfunctions]\label{lem:L1normWphi}
Let $\mathcal W$ satisfy  Assumption \ref{ass:polynome}. Then we have
 \[{\norme{\mathcal{W} \phi_k}}_{L^1(\R)} \lesssim k^{\frac{5 s+2}{2 s+2}} . \]
\end{prop}

\begin{proof} From Assumption \ref{ass:polynome} and $\lambda _k\to +\infty$, we can find $k_0\geq 0$ large enough so that the following facts hold for all $k\geq k_0$: there are $-y_k<0$ and $x_k>0$ such that \begin{eqnarray*}
\mathcal{W}(x)+\lambda_k \geq 0 &&  \forall x\in(-y_k,x_k)\\
\mathcal{W}(x)+\lambda_k=0 && \forall x\in \{-y_k,x_k\}\\
\mathcal{W}(x)+\lambda_k \leq 0 &&  \forall x\in\R\setminus (-y_k,x_k),
\end{eqnarray*}
and $\mathcal W$ is decreasing on $(-\infty,-y_k)\cup (x_k,+\infty)$. 
Assumption \ref{ass:polynome} implies that $x_k,y_k\sim \lambda _k^\frac{1}{2s}$ and thus, from Proposition \ref{prop:asymptotic}, $x_k,y_k \lesssim  k^{\frac{1}{s+1}}$. Next, up to enlarging $k_0$ if necessary, it follows from Assumption \ref{ass:polynome} that $\mathcal{W}(x)+\lambda_k\leq -1$ for all $x\in (2 x_k,+\infty)$. As a result, functions \[
\phi^\pm(x):=\pm\normeinf{\phi_k} e^{-(x-2 x_k)}\] are respectively  super and sub-solutions of the eigenvalue equation \[-\phi_k''-(\mathcal{W}(x)+\lambda_k)\phi_k=0,\] 
so that
\begin{equation}\label{loin}
\abs{\phi_k(x)}\leq \normeinf{\phi_k} e^{-(x-2x_k)}\quad \forall x\in(2x_k,+\infty),
\end{equation}
by the comparison principle. An analogous estimate holds on $(-\infty,2 y_k)$.

In order to estimate ${\norme{\mathcal{W}\phi_k}}_{L^1(\R)}$, we split the domain of integration  into three parts: $\Omega_1:=\R\setminus(-2y_k,2x_k)$, $\Omega_2:=(-2y_k,-y_k)\cup(x_k,2x_k)$ and $\Omega_3:=(-y_k,x_k)$. Setting \[
I_i := {\int}_{\Omega_i} \abs{\mathcal{W}(x)\phi_k(x)} \, dx=\int_{\Omega_i} -\mathcal{W}(x)\abs{\phi_k(x)} \, dx,
\] we decompose $\normeLq{1}{\R}{\mathcal W \phi_k}=I_1+I_2+I_3$. Notice that, 
\[\int_{2 x_k}^{+\infty} -\mathcal{W}(x) e^{-(x-2x_k)} \, dx= \left[-P(x) e^{-(x-2x_k)} \right]_{2x_k}^{+\infty}=P(2x_k),\] where $P$ is a polynomial of same degree as $\mathcal{W}$. Hence, from \eqref{loin} and Proposition \ref{prop:nomreinfiniefonctionspropres}, we get
\begin{eqnarray*}
I_1  &\lesssim& \normeinf{\phi_k} \int_{2 x_k}^{+\infty} \abs{\mathcal{W}(x)} e^{-(x-2x_k)} \, dx
  \lesssim k^{\frac{s}{2(s+1)}}P(2x_k)\\
   &\lesssim & k^{\frac {s}{2(s+1)}}x_k^{2s} \lesssim k^{\frac {s}{2(s+1)}}k^{\frac{2s}{s+1}}= k^{\frac{5 s}{2s+2}}.
\end{eqnarray*}
By monotonicity of $\mathcal{W}$ on $\Omega_2$, 
\[
I_2  \lesssim x_k \, \abs{\mathcal{W}(2 x_k)} \normeinf{\phi_k}\lesssim  x_k^{1+2s} k^{\frac{s}{2(s+1)}}\lesssim k^{\frac{5 s+2}{2 s+2}}.
\]
Remember that $-\mathcal{W}(x)\leq \lambda_k$ in $\Omega _3$ so that
\[
I_3  \lesssim  x_k \lambda _k \normeinf{\phi_k} \lesssim x_k k^{\frac{2s	}{s+1}}k^{\frac{s}{2(s+1)}}
\lesssim k^{\frac{5 s+2}{2 s+2}}.
\]
Finally, ${\norme{\mathcal{W} \phi_k}}_{L^1(\R)} \lesssim k^{\frac{5 s+2}{2 s+2}}$.
\end{proof}

\section{Well-posedness and long time behaviour}\label{s:main}

In this section we show that the Cauchy problem \eqref{eq} has a unique smooth solution which is global in time. Keystones are the change of variable \eqref{vEnFonctionDeU} that links the non-local equation \eqref{eq} to a linear parabolic problem, and our previous estimates on the underlying eigenelements. Equipped with the representation \eqref{seriessolution} of the solution, we then prove convergence in any $L^p$, $1\leq p\leq +\infty$, to the principal eigenfunction normalized by its mass. 

Up to subtracting a constant to the confining fitness function $\mathcal W$, we can assume without loss of generality (recall the mass conservation property) that $\mathcal W\leq -1$.

\subsection{Functional framework}\label{ss:framework}
For $\mathcal{W}$ a negative confining fitness function (see Assumption \ref{ass:confining}), we set
\begin{equation*}
L_{-\mathcal{W}}^2(\R) := \left\{ v:\R\to \R, \Vert v\Vert _{L_{-\mathcal W}^2(\R)}^2:={\int_\R -\mathcal{W}(x) v^2(x) \,dx} < +\infty\right\}.
\end{equation*}
Recall that the Sobolev space $W^{1,2}(\R)$ is defined as \[W^{1,2}(\R)=\{f:\R\to\R, f\in L^2(\R), f'\in L^2(\R)\},\] where the derivative $f'$ is understood in the distributional sense. We denote by $V:=W^{1,2}(\R)\cap L^2_{-\mathcal{W}}(\R)$  the Hilbert space with inner product defined by
\begin{equation*}
{(v,\vdroit)}_V :=\int_\R \frac{dv}{dx}(x) \frac{d\vdroit}{dx}(x)\, dx +\int_\R -\mathcal{W}(x) v(x) \vdroit(x)\, dx,
\end{equation*}
and $H:=L^2(\R)$ with usual inner product \[{\left(v,\vdroit\right)}_H=\int_\R v(x) \vdroit(x)\, dx.\]
By Assumption \ref{ass:confining}, it is straightforward that $L^2_{-\mathcal{W}}(\R)\subset L^2(\R)$, so  that  $V\subset H$. Moreover, the following holds.

\begin{lem}\label{lem:inj-compacte}
The embedding $V \hookrightarrow H$ is dense, continuous and compact.
\end{lem}

\begin{proof} This is very classical but, for the convenience of the reader, we present the details. Since $\mathscr{C}^{\infty}_{c} (\R) \subset V$ and $\mathscr{C}^{\infty}_{c} (\R)$ is dense in $L^2(\R)=H$, it follows that $V$ is dense in $L^2(\R)$. Next, since for $\vdroit\in V$, 
\[\int_\R \vdroit^2(x)\, dx \leq \frac{1}{- \sup \mathcal{W}} {\norme{\vdroit}}^2_V,\]
the embedding $V \hookrightarrow H$ is continuous. 

The proof of compactness follows by the Riesz-Fréchet-Kolmogorov theorem, see  e.g. \cite[Theorem 4.26]{Brezis}. Let ${(\vdroit_n)}_{n\geq 0}$ a bounded sequence of functions of $V$: there is $M>0$ such that, for all $n\geq 0$,
\begin{equation*}
\label{suite_bornee}
{\norme{\vdroit_n}}_V^{2}=\int_\R \vdroit_n '(x)^2\, dx + \int_\R -\mathcal{W}(x) \vdroit_n(x)^2\, dx < M.
\end{equation*}
We first need to show the uniform smallness of the tails of $\vdroit _n^{2}$. Let $\varepsilon>0$. Select  $a>0$ large enough so that $\frac{1}{-\mathcal{W}(x)}\leq \varepsilon$  for all $\vert x\vert \geq a$. Then
\begin{align*}
\norme{\vdroit_n}^2_{L^2(\R\setminus [-a,a])}  & =  \int_{-\infty}^{-a} \vdroit_n(x)^2\, dx + \int_a^{+\infty} \vdroit_n(x)^2\, dx \\
& \leq \varepsilon \int_{-\infty}^{-a} -\mathcal{W}(x)\vdroit_n(x)^2\, dx + \varepsilon \int_a^{+\infty} -\mathcal{W}(x)\vdroit_n(x)^2\, dx\\
&\leq  M \varepsilon.
\end{align*}
Next, for a compact set $K$, we need to show the uniform smallness of the $L^{2}(K)$ norm of $\vdroit _n(\cdot+h)-\vdroit _n$. Let $\varepsilon>0$.  By Morrey's theorem, there is $C>0$ such that, for all $h\in\R$ and $n\geq 0$,
\begin{equation*}
\abs{\vdroit_n(x+h)-\vdroit_n(x)}\leq C \abs{h}^{1/2} \normeLq{2}{\R}{\vdroit_n'}\leq C \abs{h}^{1/2} M^{1/2},
\end{equation*}
so that
\begin{equation*}
\norme{\vdroit_n(\cdot+h)-\vdroit_n}^2_{L^2(K)}   \leq  C^2 M \abs{K} \abs{h}\leq \varepsilon,
\end{equation*}
for all $n\geq 0$, if $\vert h\vert$ is sufficiently small. The lemma is proved.
\end{proof}

\subsection{Main results}\label{ss:main}

We first define the notion of solution to the Cauchy problem \eqref{eq}.

\begin{definition}[Admissible initial data]
We say that a function $u_0$ is an admissible  initial data if $u_0\in L^1(\R) \cap L^\infty(\R)$, $u_0(\cdot) \geq 0$ and $\textstyle \int_\R u_0(x)\, dx =1$.
\end{definition}

\begin{definition}[Solution of the Cauchy problem \eqref{eq}]\label{def:sol}
Let $u_0$ be an admissible initial data. We say that $u=u(t,x)$ is a (global) solution  of  the Cauchy problem \eqref{eq} if, for any $T>0$, $u\in \Ci{0}{0,T;H}\cap L^2(0,T;V)$, $u(0,\cdot)=u_0$, and 
\begin{enumerate}[label=(\roman*)]
\item For all $\vdroit\in V$, all $t\in(0,T]$,
\begin{equation*}
\begin{split}
& \frac{d}{dt} {\int_\R u(t,x)\vdroit(x)}\, dx  + \sigma^2 \int_\R \frac{\partial u}{\partial x} (t,x) \frac{d \vdroit}{d x}(x)\, dx\\
& +\int_\R -\mathcal{W}(x) u(t,x) \vdroit(x)\, dx + \overline{u}(t) \int_\R u(t,x) \vdroit(x) \, dx =0,
\end{split}
\end{equation*}
where the time derivative is understood in the distributional sense. Equivalently, for all $\vdroit\in V$, all $\varphi\in \mathscr{C}^1_{\rm c}(0,T)$,
\begin{equation}\label{WeakFormulation}
\begin{split}
&-\int_0^T  \left( {\int_\R u\vdroit}\, dx\right) \, \varphi'(t)\, dt  + \sigma^2 \int_0^T \left(\int_\R \frac{\partial u}{\partial x}  \frac{d \vdroit}{d x}\, dx \right) \varphi(t)\, dt  \\ & + \int_0^T \left( \int_\R -\mathcal{W} u\vdroit\, dx \right) \varphi(t)\, dt  + \int_0^T  \left( \int_\R u \vdroit \, dx\right)\overline{u}(t) \varphi(t) \, dt=0.
\end{split}
\end{equation}
\medskip 
\item $\overline{u}:t\mapsto \textstyle \int_\R \mathcal{W}(y) u(t,y)\, dy$ is a continuous function on $(0,T]$.
\medskip
\item  $C_T:=\textstyle \int_0^T \abs{\overline{u}(t)}\, dt < +\infty$.
\end{enumerate}
\end{definition}

\medskip

Here is our main mathematical result.

\begin{thm}[Solving replicator-mutator problem]\label{thm:main}
Let $\mathcal{W}$ satisfy Assumption \ref{ass:polynome}.  For any admissible initial condition $u_0$, there is a unique solution $u=u(t,x)$ to the Cauchy problem \eqref{eq}, in the sense of Definition \ref{def:sol}. Moreover the solution is smooth on $(0,+\infty)\times \R$  and is  given by
\begin{equation}\label{seriessolution}
u(t,x)=\frac{\displaystyle \sum_{k=0}^{+\infty} {\left(u_0,\phi_k \right)}_{L^2(\R)} \phi_k(x) e^{-\lambda_k t}}{\displaystyle \sum_{k=0}^{+\infty} {\left(u_0,\phi_k \right)}_{L^2(\R)} m_k e^{-\lambda_k t} }, \quad t>0,\, x\in \R,
\end{equation}
where $(\lambda _k,\phi _k)$ are the eigenelements defined in Proposition \ref{prop:eigen}, and \[m_k= \int _\R \phi_k(x)\,dx.\]
\end{thm}

\begin{proof}
We proceed by necessary and sufficient condition. Let $u$ be a solution, in the sense of  Definition \ref{def:sol}. We define the function $v$ as
\begin{equation}\label{vEnFonctionDeU}
v(t,x):= u(t,x) \exp{\left(\int_0^t \overline{u}(\tau)\, d\tau \right)}, \quad 0\leq t\leq T, \, x\in \R.
\end{equation}
This function is well defined since by Definition \ref{def:sol} $(iii)$, the integral in the exponential is finite for all  $t\in[0,T]$. Since $C_T<\infty$ and $u(t,\cdot)\in H\cap V$, it is straightforward to see that, for all $t\in(0,T]$, $v(t,\cdot)\in H\cap V\equiv V$. Additionally, from  $C_T<\infty$ and $u\in\Ci{0}{0,T;H}$, one can see  that $v\in \Ci{0}{0,T;H}$. Last, $v\in L^2(0,T;V)$ due to
\begin{align*}
\int_0^T {\norme{v(t,\cdot)}_V^2} \, dt  & = \int_0^T {\norme{u(t,\cdot)\umean}_V^2} \, dt \leq e^{2 C_T}  \int_0^T {\norme{u(t,\cdot)}_V^2}\, dt <  \infty,
\end{align*}
since $u\in L^2(0,T;H)$. 

\medskip

We  now show that $v$ solves the linear Cauchy problem 
\begin{equation}\label{problemeLineairestar}
\left\{\begin{array}{ll}
\partial_t v =  \sigma^2 \partial^2_x v + \mathcal{W}(x) v \vspace{5pt} \\
v(0,x) =u_0(x).
\end{array}
\right.\end{equation}
Indeed, formally for the moment, 
\begin{equation*}
\begin{split}
& \partial_t v = (\partial_t u ) \umean + u \overline{u}\\ 
&\partial^2_x v = (\partial^2_x u) \umean,
\end{split}
\end{equation*}
so that
\begin{equation*}
\label{testFunction}
\partial_t v - \sigma^2 \partial ^2_x v -\mathcal{W}(x) v =  (\partial_t u + u \overline{u} - \sigma^2 \partial^2_x u -\mathcal{W}(x) u) \umean =0
\end{equation*}
since $u$ solves \eqref{eq}. Those computations can be made rigorous in the  distributional sense. Indeed, for a test function $\psi\in \ftest$,  set
\begin{equation*}\label{etoilemarroquaine}
\varphi(t):=\psi(t) \umean,
\end{equation*}
and by  Definition \ref{def:sol} $(ii)$,   $\varphi$ belongs to $\mathscr{C}^1_{c}(0,T)$. Writing \eqref{WeakFormulation} with $\varphi$ as test function yields the weak formulation of \eqref{problemeLineairestar} with $\psi$ as test function, that is
\begin{equation*}
\begin{split}
&-\int_0^T  \left( {\int_\R v(t,x)\vdroit(x)}\, dx\right) \, \psi'(t)\, dt  + \sigma^2 \int_0^T \left(\int_\R \frac{\partial v}{\partial x} (t,x) \frac{d \vdroit}{d x}(x)\, dx \right) \psi(t)\, dt  \\ & + \int_0^T \left( \int_\R -\mathcal{W}(x) v(t,x) \vdroit(x)\, dx \right) \psi(t)\, dt =0,
\end{split}
\end{equation*} 
for all $\vdroit \in V$.

\medskip

The well-posedness of the linear Cauchy problem \eqref{problemeLineairestar} is postponed to the next subsection: from Proposition \ref{prop:existenceParabolicProblem}, we know that, for all $t\in(0,T]$,%
\begin{equation*}
v(t)=\sum_{k=0}^{+\infty} {\left(u_0,\phi_k \right)}_{L^2(\R)} \phi_k e^{-\lambda_k t} \quad \text{ in } L^{2}(\R).
\end{equation*}
Now, the estimates on the eigenvalues and the $L^\infty$ norm of eigenfunctions, namely Proposition \ref{prop:asymptotic} and Proposition  \ref{prop:nomreinfiniefonctionspropres}, allow to write
\begin{equation*}\label{v:ponctuel}
v(t,x) = \sum_{k=0}^{+\infty} {\left(u_0,\phi_k \right)}_{L^2(\R)} \phi_k(x) e^{-\lambda_k t}, \quad 0<t\leq T, \, x\in \R.
\end{equation*} 
Also, we know from the parabolic regularity theory and the comparison principle, that  $v\in \mathscr{C}^\infty((0,T)\times\R)$ and that $v(t,x)>0$ for all $t>0$, $x\in \R$.

Now, we show that the change of variable  \eqref{vEnFonctionDeU} can be inverted. For $t>0$, multiplying \eqref{vEnFonctionDeU} by $\mathcal{W}(x)$ and integrating over $x\in\R$, we get
\begin{align}
\overline v(t):=\int_\R \mathcal{W}(x) v(t,x)\, dx & = \overline{u}(t) \exp{\left(\int_0^t \overline{u}(\tau)\, d\tau \right)}  \nonumber \\
& = \frac{d}{dt} \left( \exp{\left(\int_0^t \overline{u}(\tau)\, d\tau \right)} \right).\label{eq1}
\end{align}
On the other hand, we claim that, for all $t>0$,
\begin{equation}
\label{eq2}
\frac{d}{dt}m_v(t)=\overline v(t),
\end{equation} 
which follows formally by integrating \eqref{problemeLineairestar} over $x\in \R$. To prove \eqref{eq2} rigorously, notice first that by Proposition \ref{prop:asymptotic} and \ref{prop:mass-estimate}, the series
\[
\sum_{k=0}^{+\infty} \abs{{\left(u_0,\phi_k \right)}_{L^2(\R)}} e^{-\lambda_k t} \int_{\R}\abs{\phi_k(x)}\, dx
\]
converges for all $t>0$. Hence $\textstyle m_v(t)$, the total mass of $v$, is given by \[ m_v(t) = \sum_{k=0}^{+\infty} {\left(u_0,\phi_k \right)}_{L^2(\R)} m_k e^{-\lambda_k t}.\]
Next, for any $t_0>0$ , $\sum_{k=0}^{+\infty} \abs{{\left(u_0,\phi_k \right)}_{L^2(\R)}} \abs{m_k } \lambda_k e^{-\lambda_k t_0}<+\infty$ thanks to Proposition  \ref{prop:asymptotic} and \ref{prop:mass-estimate}, so that $m_v$ is differentiable on $(0,T]$ and 
\[\frac{d}{dt }m_v (t) =  \sum_{k=0}^{+\infty} {\left(u_0,\phi_k \right)}_{L^2(\R)} m_k (-\lambda_k) e^{-\lambda_k t} =\overline{v}(t),\] the last equality following by similar arguments based on Proposition \ref{lem:L1normWphi}. Hence \eqref{eq2} is proved. From \eqref{eq1}, \eqref{eq2} and $m_v(0)=1$, we deduce that
\[
\exp{\left(\int_0^t \overline{u}(\tau)\, d\tau \right)}=m_v(t),
\]
for all $t\geq 0$. As a conclusion, \eqref{vEnFonctionDeU} is inverted into
\begin{equation}\label{uegalv}
u(t,x)=\frac{v(t,x)}{m_v(t)}, \quad m_v(t):= \int_\R v(t,x)\, dx>0,
\end{equation}
for all $0\leq t\leq T$, $x\in \R$.

\medskip

Conversely, we need to show that the function $u$ given by \eqref{uegalv} is  the solution of \eqref{eq} in the sense of Definition \ref{def:sol}. Let $T>0$.  

Since $\overline{u}(t)=\frac{\overline{v}(t)}{m_v(t)}=\frac{\frac{d}{dt}{m}_v(t)}{m_v(t)}$, the function $\overline{u}$ is continuous on $(0,T]$, which shows item $(ii)$ of Definition \ref{def:sol}.

Next, since $m_v>0$ and $\overline{v}<0$, 
\[\int_0^T \abs{\overline{u}(t)}\, dt  = -\int_0^T \frac{\frac{d}{dt}{m}_v(t)}{m_v(t)}\, dt=-\ln(m_v(T)) <+\infty,
\]
which shows item $(iii)$ of Definition \ref{def:sol}.

Last, since $v\in \Ci{0}{0,T;H}\cap L^2(0,T;V)$  and $0<m_v(T)\leq m_v(t)\leq 1$ for any $0\leq t\leq T$, then $u \in \Ci{0}{0,T;H}\cap L^2(0,T;V)$. For a test function $\varphi \in \ftest$,  set%
\begin{equation*}\label{etoilemarroquaine2}
\psi(t):=\varphi(t) \umeaninverse,
\end{equation*}
writing the weak formulation of \eqref{problemeLineairestar} with $\psi$ as test function, we see that $u$ given by \eqref{uegalv} satisfies the weak formulation \eqref{WeakFormulation} with $\varphi$ as test function, which shows item $(i)$ of Definition \ref{def:sol}. 

Theorem \ref{thm:main} is proved. 
\end{proof}

We are now in the position to understand the long time behaviour of the solution, of crucial importance for the biological interpretation (branching phenomena) in Section \ref{s:branching}.

\begin{cor}[Long time behaviour]\label{cor:long-time}
Let $\mathcal W$ satisfy Assumption \ref{ass:polynome}. Let $u_0$ be an admissible initial data $u_0$. Then the solution $u$ to the Cauchy problem \eqref{eq} converges, at large time, to the ground state $\phi _0$ divided by its mass $m_0:=\textstyle\int_\R \phi _0(x)\, dx$. Precisely, for any $1\leq p\leq +\infty$,
\[
u(t,\cdot)-\frac{\phi_0(\cdot)}{m_0}\longrightarrow 0 \quad \text{ in } L^p(\R), \quad \text{ as } t \to+\infty.
\]
\end{cor}

\begin{proof} We denote $a_k:=(u_0,\phi _k)_{L^{2}(\R)}$ and observe that 
$a_0>0$ since $\phi_0>0$ and $u_0\geq 0$, $u_0\not \equiv 0$. Thus, from \eqref{seriessolution} we have
\[
u(t,x)  = \frac{\displaystyle \phi_0(x)+\frac{1}{a_0}\sum_{k=1}^{+\infty} a_k \phi_k(x) e^{-(\lambda _k-\lambda _0) t} }{\displaystyle m_0+\frac{1}{a_0}\sum_{k=1}^{+\infty} a_k m_k e^{-(\lambda _k-\lambda_0) t}}.
\]
Recall that $\lambda_0< \lambda_1 \leq \lambda_k$ for all $k\in\N^*$ and that we are equipped with the asymptotics of Proposition \ref{prop:asymptotic}. Hence, by Proposition \ref{prop:mass-estimate} and the dominated convergence theorem, the  denominator tends to $m_0$ as $t\to +\infty$.  Similarly, by Proposition \ref{prop:nomreinfiniefonctionspropres}, Proposition \ref{prop:mass-estimate} respectively, the numerator tends to $\phi_0$ in $L^{\infty}(\R)$, $L^{1}(\R)$ respectively, as $t\to +\infty$. For $1<p<+\infty$, the result follows by interpolation.
\end{proof}

In particular, Corollary \ref{cor:long-time} implies that, whatever the number of maxima of the initial date $u_0$, the long time shape is determined by that of the ground state $\phi _0$.  We illustrate this property with numerical simulations in Section \ref{s:branching}.

Note that Corollary \ref{cor:long-time} is an extension of  the long time convergence result proved in \cite{AlfaroCarles14}, for to the particular case of a quadratic fitness, $\mathcal{W}(x)=-x^2$, for which it is well known that the  principal eigenfunction is a Gaussian.

\subsection{Linear parabolic equation}\label{ss:linearParabolicEqn} For the convenience of the reader, we recall here how to deal with the linear Cauchy problem \eqref{problemeLineairestar}.

\begin{prop}[The linear problem]\label{prop:existenceParabolicProblem}
For any $u_0\in \Lq{1}{\R}\cap\Linf{\R}$, for any $T>0$, the problem \eqref{problemeLineairestar} posses a unique weak solution $v=v(t,x)$, in the sense that $v\in \Ci{0}{0,T;H}\cap L^2(0,T;V)$, $v(0,\cdot)=u_0$ and for all $\vdroit\in V$, all $t\in(0,T]$, 
\[ \frac{d}{dt} \int_\R v(t,x) \vdroit(x)\, dx + \sigma^2 \int_\R \partial_x v(t,x) \partial_x \vdroit(x) +\int_\R -\mathcal{W}(x) v(x) \vdroit(x)\, dx = 0, \]
where the time derivative is understood in the distributional sense. Furthermore, \begin{equation*}\label{solutionSeries}
v(t)=\sum_{k=0}^{+\infty} {\left( u_0, \phi_k\right)}_{L^2(\R)} \phi_ke^{-\lambda_k t}
 \end{equation*}
and the convergence of the sequence of partial sums is uniform in time. 
\end{prop}

\begin{proof} The form $\mathbf{a}:{V\times V}\to {\R}$ 
 \[
 \mathbf{a}(v,\vdroit):={\sigma^2 \int_\R \frac{dv(x)}{dx} \frac{d\vdroit(x)}{dx} +\int_\R -\mathcal{W}(x) v(x) \vdroit(x)\, dx}
 \]
 is   symmetric and  bilinear. It is continuous since, for all $v,\vdroit\in V$,
\begin{align*}
\abs{\mathbf a(v,\vdroit)} & \leq \sigma^2 \abs{\int_\R v'(x) \vdroit'(x)\, dx }+\abs{\int -\mathcal{W}(x) v(x) \vdroit(x)\, dx}\\
& \leq  \sigma^2 \normeLq{2}{\R}{v'}\normeLq{2}{\R}{\vdroit'}+ {\norme{v}_{L^2_{-\mathcal{W}}(\R)}} {\norme{\vdroit}_{L^2_{-\mathcal{W}}(\R)}}\\
& \leq ( \sigma^2 +1) \norme{v}_V \norme{\vdroit}_V.
\end{align*}
It is coercive since, for all $\vdroit\in V$,
\begin{align*}
\mathbf a(\vdroit,\vdroit) & =  \sigma^2 \int_\R (\vdroit'(x))^2\, dx +\int_\R -\mathcal{W}(x) \vdroit^2(x)\, dx\\
& \geq \min( \sigma^2,1) {\norme{\vdroit}}_V^2.
\end{align*}
The conclusion then follows from Lemma \ref{lem:inj-compacte}
 and Lions' Theorem  for parabolic equations. \end{proof}

We state Lions' theorem covering parabolic Cauchy problems of the form 
\begin{equation*}
\left\{\begin{array}{ll}
\partial_t v =  \sigma^2 \partial^2_x v + \mathcal{W}(x) v +f(t,x)\vspace{5pt} \\
v(0,x) =u_0(x).
\end{array}
\right.\end{equation*}

\begin{thm}[Lions' theorem, see  \cite{lions_problemes_1968} or \cite{rakotoson}]\label{th:Lions}
Let $V$ be a separable Hilbert space with inner product ${(\cdot ,\cdot)}_V$  and norm ${\norme{\cdot}}_V$. Let $H$ be a Hilbert space with inner product ${(\cdot,\cdot)}_H$ and norm ${\norme{\cdot}}_H$, such that  $H\simeq H'$. Assume that the  embedding $V \hookrightarrow H$ is dense, continuous and compact. Let $\mathbf{a}:V\times V\rightarrow \R$ be a symmetric, continuous and coercive bilinear form. Let $T>0$ and  $f \in L^2(0,T;H)$ be given. Let $u_0\in H$ be given.

Then, there is a unique function $v\in \Ci{0}{0,T;H}\cap L^2(0,T;V)$ such that, for all $\vdroit\in V$,
\[\left\{ 
\begin{array}{ll}
 \frac{d}{dt} {(v(t),\vdroit)}_H+\mathbf{a}(v(t),\vdroit)={\langle f(t),\vdroit\rangle}_{V',V} \quad {\rm in} \ \mathscr{D}'(0,T) \vspace{5pt}\\
 v(0)=u_0.
\end{array}
\right.
\]
Moreover, for all $t\in [0,T]$, $v$ is written as the Hilbertian sum
\[
 v(t)=\sum_{j=0}^{+\infty} g_j(t) \phi_j,
\]
 where $(\phi_j)_{j\geq 0}$ is the spectral basis of $H$ defined by $\mathbf a (\phi_j,\vdroit)=\lambda_j(\phi _j,\vdroit)$ for all $\vdroit \in V$,
\[
g_j(t):={(u_0,\phi_j)}_H e^{-\lambda_j t}+\int_0^t f_j(s) e^{-\lambda_j (t-s)}\, ds, \quad f_j(s):={(f(s),\phi_j)}_H.
\]
Also, the sequence  $v_n(t):= \sum_{j=0}^n g_j(t)\phi_j$ uniformly converges to $v(t)$, that is
\[
\sup_{0\leq t\leq T}\Vert v_n(t)-v(t)\Vert _H \to  0, \quad \text{ as } n\to +\infty.\]
\end{thm}

\section{Branching or not}\label{s:branching}

\emph{Evolutionary branching}  is a corner stone in the theory of evolutionary genetics  \cite{book:832770}, \cite{book:1324333}. It consists in the splitting from uni-modal to multi-modal distribution of the phenotypic trait.  
By Corollary \ref{cor:long-time}, if the principal eigenfunction $\phi_0$ has two or more maxima, it follows that for any uni-modal initial condition $u_0$, the solution will split to multi-modal distribution in the limit $t\to +\infty$. From a biological point of view, the fitness function $\mathcal{W}$ is the key element for branching to occur. However, the mutation rate $\sigma$ is another main parameter involved in the branching process. Indeed, if $\sigma$ is large, then the population distribution tends to homogenize. As a consequence, a too large value of $\sigma$ may enforce uni-modality of 
the principal eigenfunction $\phi_0$ and thus of the solution $u(t,\cdot)$ as $t$ tends to $+\infty$. In this section, we enquire on the conjunct influence of $\mathcal W$ and of $\sigma$ on the shape of the principal eigenfunction $\phi _0$. This is far from being straightforward, and we therefore combine some rigorous results and  numerical explorations. We start with some particular cases of analytic ground states.

\subsection{Some explicit ground states}\label{ss:explicit} The search for eigenvalues and  eigenfunctions of  Schrödinger operators has long been  motivated by applications in physics and chemistry. Some closed-form formulas of eigenfunctions, in particular the ground state, for specific potentials are available in the literature. For example in  \cite{Brandon2013}, the authors get for $\sigma=1$, 
\begin{equation*}
\begin{cases}
-\mathcal{W}(x) =x^{10}-x^8+x^6-\tfrac{43 }{8}x^4+\tfrac{105}{64}x^2\, \\
 \phi_0(x)  =A\exp\left(-\frac 3{16}x^{2}+\frac 1 8 x^{4}-\frac 1 6 x^{6}\right)\\
 \lambda_0 =\tfrac{3}{8},
\end{cases}
\end{equation*}
with $A>0$ a normalization constant. In this case the potential $-\mathcal W$ is symmetric and double-well shaped, but the ground state is uni-modal because of a too large $\sigma$.

Xie, Wang and Fu \cite{XieWangFu} provide exact solutions for a class of rational potentials using the confluent Heun functions: for any $\omega>0$, $g>0$ and $V_2<g$, they obtain
\begin{equation*}
\left\{
\begin{array}{ll}
  -\mathcal{W}(x) =\frac{\omega ^2}{4}x^{2}+\frac{g (g-V_2)+g \omega +\sqrt{g (g-V_2)} (g+\omega )}{g \left(1+g x^2\right)}+\frac{V_2}{\left(1+g x^2\right)^2}\\
\phi_0(x) =A\exp\left(-\frac{ \omega }{4}x^2+\tfrac{g+\sqrt{g (g-{V_2})}}{2 g} \ln\left( 1+g x^2\right) \right)\\ 
 \lambda_0 =\left(\tfrac{\sqrt{g (g-{V_2})}}{g}+\tfrac{3}{2}\right) \omega\\
  \sigma=1.
\end{array}
\right.
\end{equation*}
Here, the potential is symmetric. At least for some parameters, see \cite{XieWangFu},  both the potential and the ground state are double-well shaped (branching occurs).

In \cite{Zaslavski}, authors give some explicit formulas for potentials defined by trigonometric hyperbolic functions: for any $B>0$ and $C\geq 0$, they obtain 
\begin{equation*}
\left\{
\begin{array}{ll}
  -\mathcal{W}(x) =\tfrac{B^2}{4}  \left(\sinh (x)-\tfrac{C}{B}\right)^2-B \cosh (x) \vspace*{5pt}\\
  \phi_0(x) =A (e^{\tfrac{x}{2}}-\tfrac{1}{B}(C-\sqrt{B^2+{C}^2}) {e^{-\tfrac{x}{2}}
   }
   ) e^{\frac{C}{2}x-\frac{B}{2} 
   \cosh (x)}\vspace{5pt}\\
 \lambda_0 =-\frac{1}{2} \sqrt{B^2+C^2}-\frac{1}{4}\\
 \sigma=1.
\end{array}
\right.
\end{equation*}
In particular, when $C=0$ the potential is symmetric and is a double-well for $0<B<2$ but a single-well when $B>2$; on the other hand the ground state has then two local maxima for $0<B<1/2$ but only one when $B> 1/2$. If we slightly increase $C>0$, thus breaking the symmetry of the potential, and keep $B>0$ small, we see that a second local maximum appears in the ground state. This shows that the shape of the ground state is very sensitive to the symmetry or not of the potential. 

To conclude this subsection, let us observe that the ansatz $\phi_0(x):=e^{-q(x)}$ is positive and satisfies $-\phi _0''-(q''(x)-(q'(x))^2)\phi_0=0$. It is therefore the ground state associated with $\sigma=1$, $\mathcal W(x)=q''(x)-(q'(x))^{2}$, $\lambda_0=0$. This provides a way to construct many examples.

\subsection{Obstacles to branching}\label{ss:no-branching} As already mentioned above, a too large $\sigma$ prevents the branching phenomenon. 

Next, if the fitness $\mathcal W$  is concave it is known \cite[Theorem 6.1]{Bra-Lie-76}, see also \cite{Hel-Sjo-94}, that the ground state  is log-concave and therefore uni-modal. For instance, the harmonic potential $-\mathcal{W}(x)=x^2$ has the ground state  $\phi_0(x)=\tfrac{1}{\sqrt{\pi \sigma}}\exp\left({-\tfrac{x^2}{2 \sigma }}\right)$ which is log-concave.

Slightly more generally, if the fitness has a unique global maximum, it is expected that, whatever the values of $\sigma$, the ground state remains uni-modal, see Figure \ref{fig:uni-modal} for numerical simulations.

\begin{figure}[hbt!]
\begin{center}
\includegraphics[width=12.5cm]{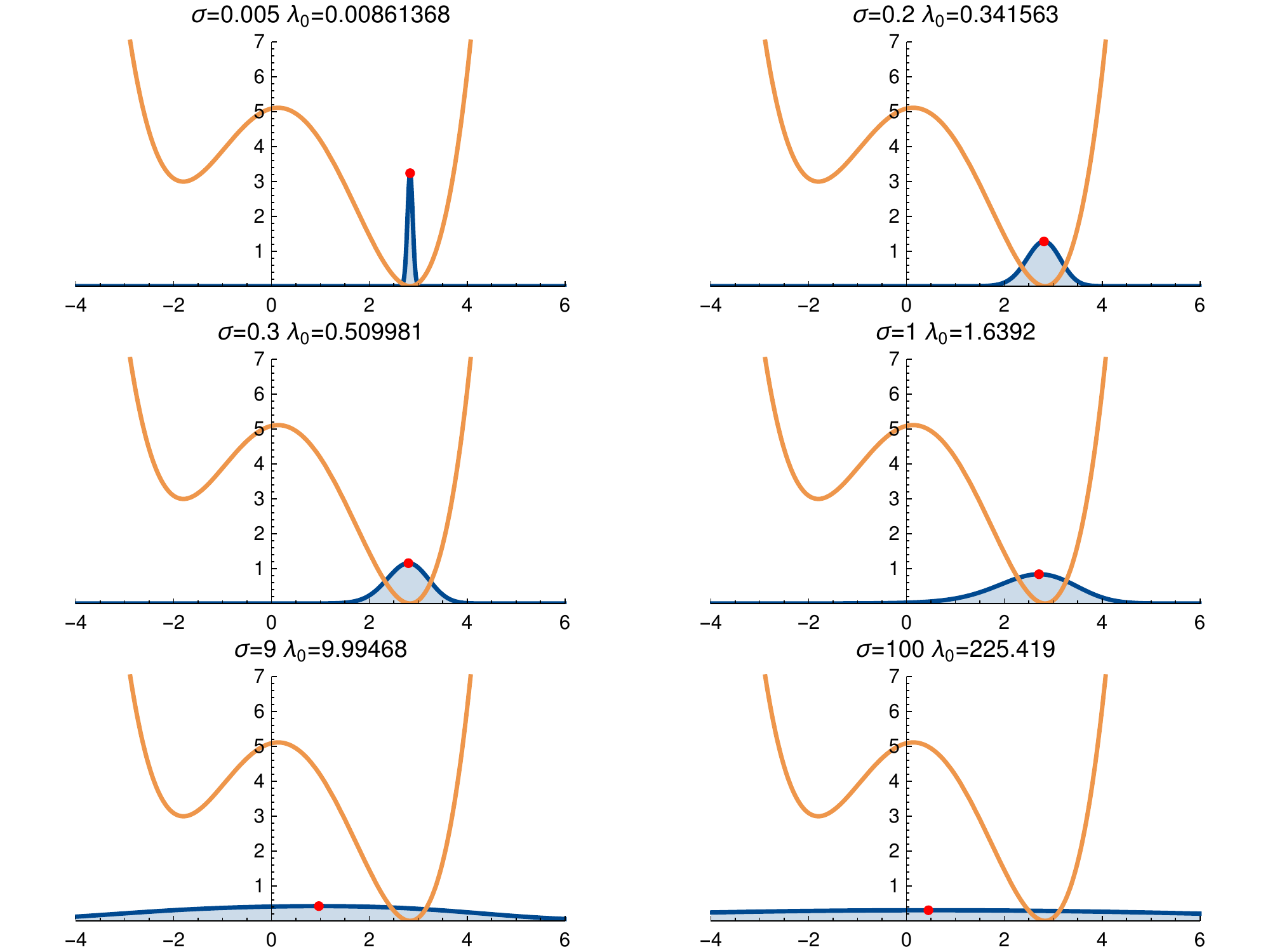}
\caption{Ground state $\phi_0(\sigma)$, increasing parameter $\sigma $, in the case $-\mathcal{W}(x)=\tfrac{299}{2520}x^{4}-\tfrac{233}{1260}x^{3}-\tfrac{2971}{2520}x^{2}+\tfrac{139}{420}x+{\rm constant}$. The ground state  remains uni-modal, and the global maximum is shifted towards left, asymptotically reaching 0.}
\label{fig:uni-modal}
\end{center}
\end{figure}

\subsection{The typical situation leading to branching}\label{ss:prototype} In order to obtain branching, the above considerations drive us to consider a fitness function $\mathcal W$ reaching multiple times its global maximum combined with a small enough parameter $\sigma>0$. Hence, in the particular case of a double-well potential $-\mathcal W$,  it is proved in \cite[Theorem~2.1]{SimonAsymptoticEigenvalue} that, far from the minima of the potential (in particular between the two wells), the ground state $\phi_0=\phi_0(\sigma)$ is exponentially small as $\sigma\to0^+$, which indicates that branching occurs.

Nevertheless,  one can come to a similar conclusion through direct arguments under the assumption that the fitness function $\mathcal W$ is even, satisfies  Assumption \ref{ass:polynome} and  $\mathcal{W}(0)<\max \mathcal{W}$. Indeed, since $\mathcal{W}$ is even, so is the ground state and therefore $\phi _ {0}'(0)=0$. Next, testing the equation at $x=0$, we get
\[
\sigma ^2\phi_0''(0)=-(\mathcal{W}(0)+\lambda _0)\phi_0(0).
\]
We know from Lemma \ref{prop:sigma-pt} that $\lambda _0=\lambda_0(\sigma)\to -\max \mathcal{W}$ as $\sigma \to 0$, and therefore, for $\sigma$ sufficiently small, $\phi _0''(0)>0$. This shows that the ground state is at least bi-modal. For instance, in Figure \ref{fig:hbartozero} we show $-\mathcal{W}(x):=\frac 1{12}(x^{2}-2)^{2}$ and the associated principal eigenfunction $\phi_0=\phi_0(\sigma)$, for different values of the parameter $\sigma$. 

\begin{figure}[hbt!]
\begin{flushleft}
\includegraphics[width=12.5cm]{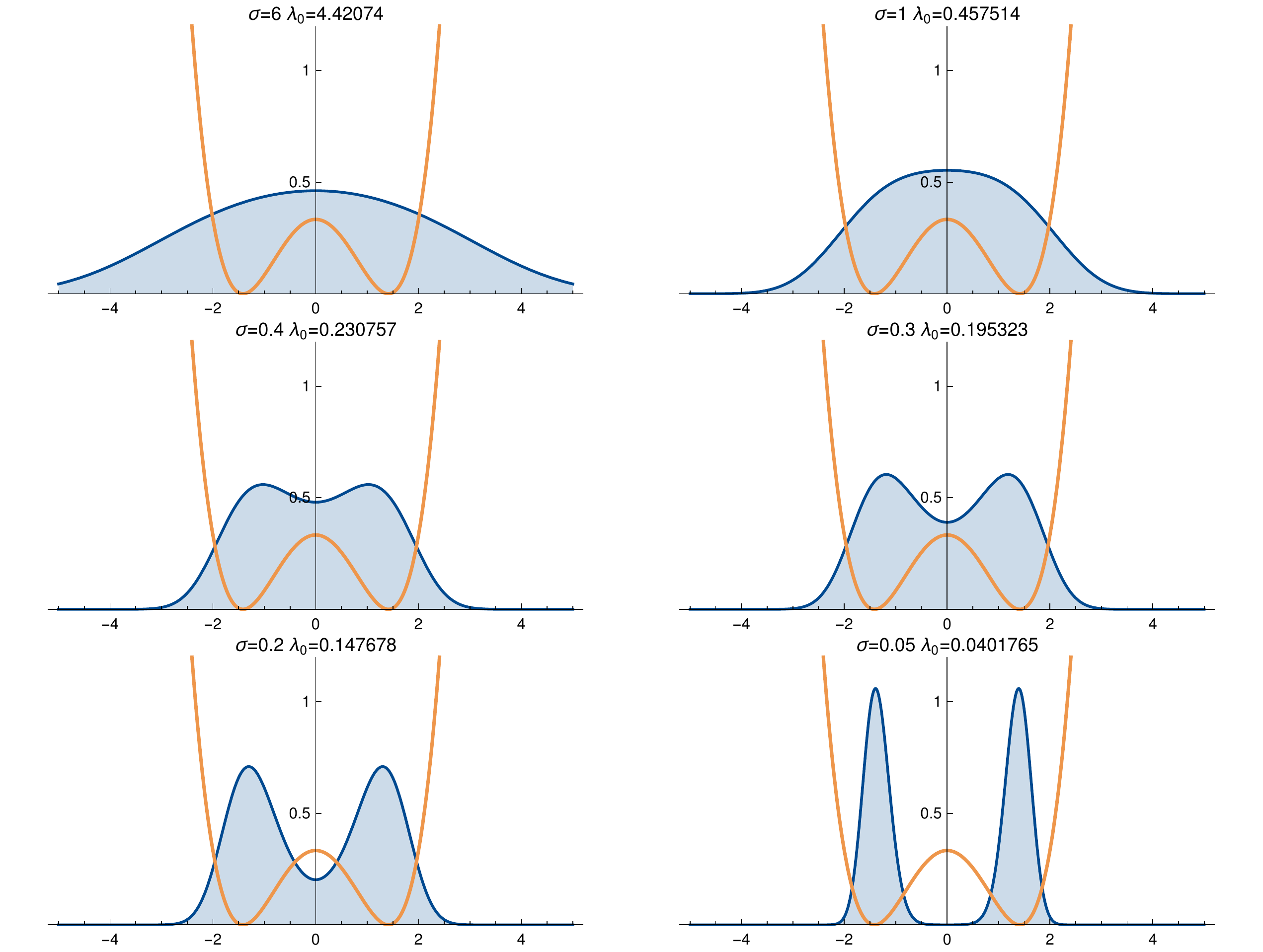}
\caption{Ground state $\phi_0(\sigma)$, decreasing parameter $\sigma $, in the case $-\mathcal{W}(x)=\tfrac 1{12}(x^{2}-2)^{2}$. Top left: for large $\sigma$ the ground state is uni-modal. Bottom right: for small $\sigma$ the ground state tends to concentrate into the wells of $-\mathcal W$.}
\label{fig:hbartozero}
\end{flushleft}
\end{figure}

As far as the Cauchy problem \eqref{eq} is concerned, we present here 
some numerical simulations where one can observe the branching phenomenon. For this example we choose  the double-well fitness function
\begin{equation}
\label{fitness-evolution}
-\mathcal{W}(x)=\left(x^2-4\right) x^2+4,
\end{equation}
and $\sigma=10^{-3}$, which is sufficiently small to ensure that $\phi_0$ is bi-modal. To numerically compute the solution $u(t,x)$ to \eqref{eq} we follow the proof of Theorem \ref{thm:main}: using finite element method we  first compute a numerical approximation $v_{\rm num}(t,x)$ of the solution $v(t,x)$ to the linear Cauchy problem \eqref{problemeLineairestar}; next, using standard quadrature methods,  we compute the mass $m_{v_{\rm num}}(t)$ of the numerical approximation $v_{\rm num}(t,x)$; last, we use the relation \eqref{uegalv} to obtain $u_{\rm num}(t,x)$. The results are plotted for different times in Figure \ref{fig:evolution-gauss-centree} and Figure \ref{fig:evolution-gauss-excentree}. In Figure \ref{fig:evolution-gauss-centree}, we use the Gaussian initial condition $u_0(x)=\tfrac{1}{\sqrt{\pi }}{e^{-x^2}}$. As for Figure \ref{fig:evolution-gauss-excentree}, we use
\[
u_0(x)=\frac{1}{\sqrt \pi}\frac{e^{-(x-4)^{2}}+\varepsilon e^{-x^{2}}}{1+\varepsilon}, \quad \varepsilon=10^{-2},
\]
the role of $\varepsilon$ being to avoid some numerical instabilities. The initial condition lies on the right of the two wells. The solution remains not symmetric but, gradually, it converges to the symmetric ground state, with some possible transient complicated patterns.

\begin{figure}[hbt!]
\begin{center}
\includegraphics[width=12.5cm]{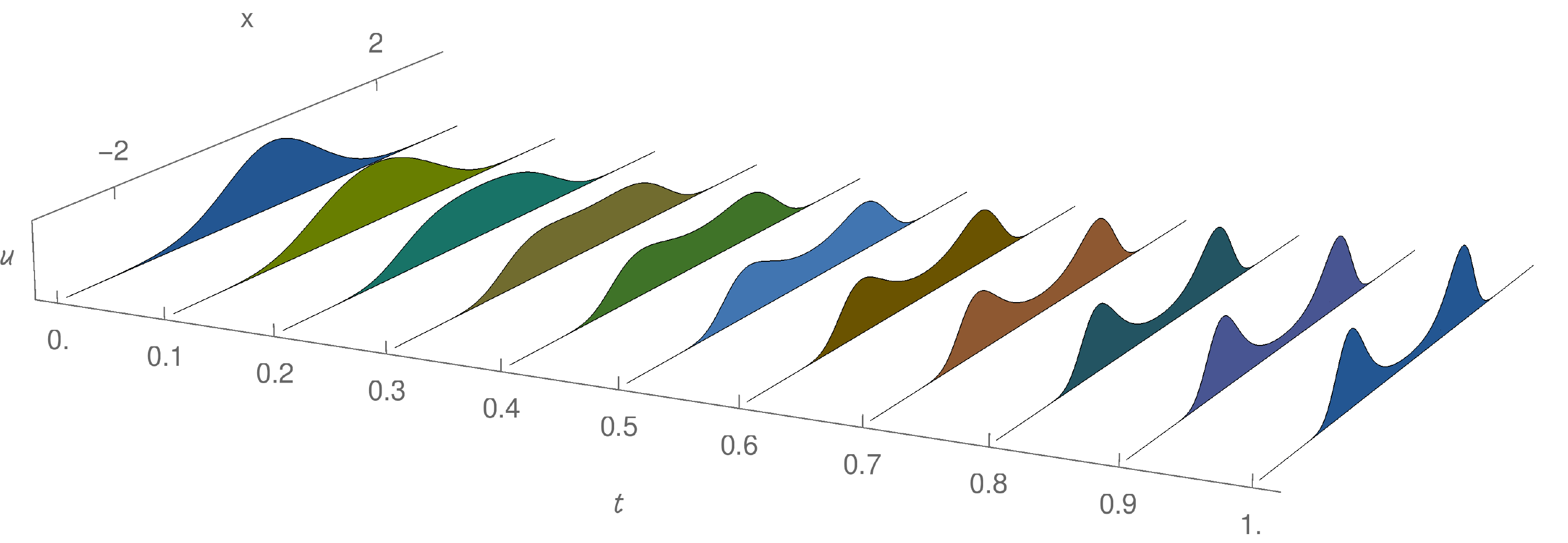}
\caption{Numerical solution of the Cauchy problem \eqref{eq} exhibiting  a branching phenomenon. Here $\sigma=10^{-3}$, fitness is as in \eqref{fitness-evolution} and the initial data is centered.} 
\label{fig:evolution-gauss-centree}
\end{center}
\end{figure}

\begin{figure}[hbt!]
\begin{center}
\includegraphics[width=12.5cm]{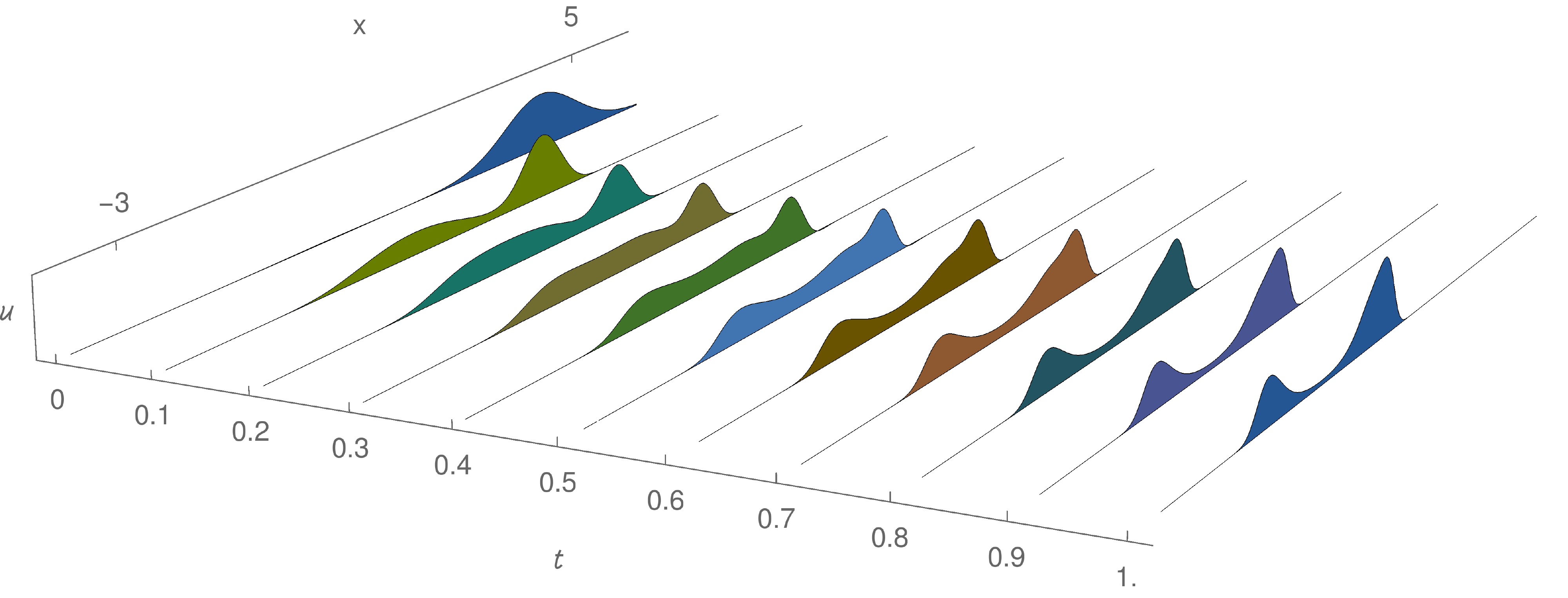}
\caption{Numerical solution of the Cauchy problem \eqref{eq} exhibiting  a branching phenomenon. Here $\sigma=10^{-3}$, fitness is as in \eqref{fitness-evolution} and the initial data is off-centered.} 
\label{fig:evolution-gauss-excentree}
\end{center}
\end{figure}

\subsection{The number of modes}\label{ss:complex} When the fitness function (assumed to be symmetric) reaches its global maximum at $N\geq 2$ points, say $x_1<\cdots<x_N$, it is expected \cite{Dji-Duc-Fab-17} that, as $\sigma \to 0$, the ground state concentrates in the $x_i$ points where the biological niche is the widest since, at these points, individuals suffer less when their traits are slightly changed by mutations.  Mathematically this means that $\phi _0$ is $p$-modal where
\[
p:=\#\left\{1\leq i \leq N: \vert \mathcal W'' (x_i)\vert=\min_{1\leq j\leq N}\vert \mathcal W'' (x_j)\vert \right\}.
\]

As a first example, consider the symmetric, triple-well potential 
\begin{equation*}
\label{potentielGordaEnMedio}
-\mathcal{W}(x)=\tfrac{1}{200} x^4 (6 x-8)^2 (6 x+8)^2,
\end{equation*}
whose wells are localized at $0$ and $\pm\tfrac{4}{3}$. The well at zero is wider than the two other ones. In this case, the ground state is, as explained above,  uni-modal for small $\sigma$. Moreover in this ``narrow-wide-narrow'' situation, the ground state remains uni-modal when we increase $\sigma$, as  numerically observed in Figure \ref{fig:pic-fGf}.
\begin{figure}
\begin{flushleft}
\includegraphics[width=12.5cm]{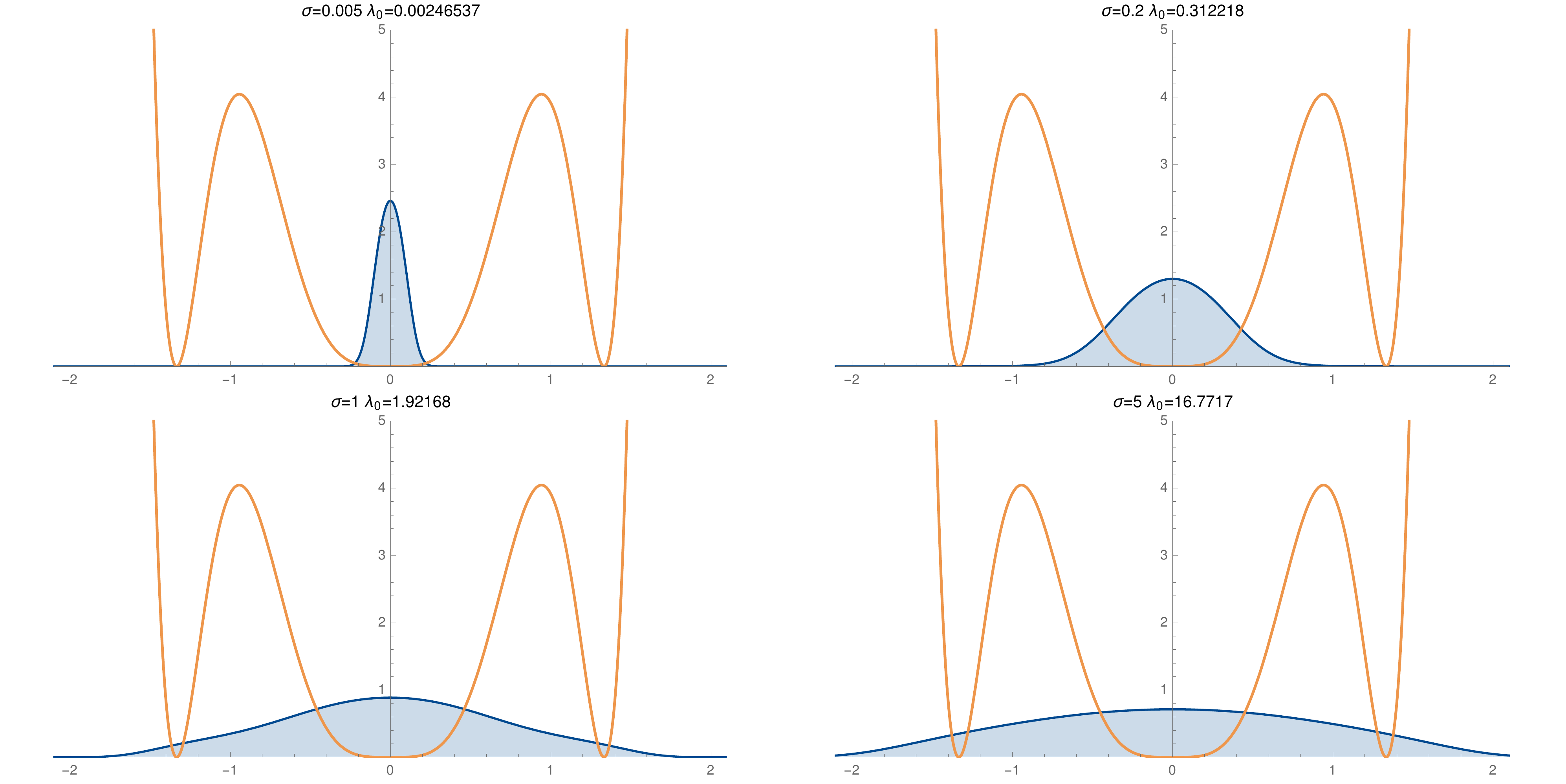}
\caption{Ground state $\phi_0(\sigma)$, increasing parameter $\sigma$, in the case 
of a ``narrow-wide-narrow'' potential.}
\label{fig:pic-fGf}
\end{flushleft}
\end{figure}

On the other hand, as the bifurcation parameter $\sigma$ increases, it may happen that, because of the position of the wells, the number of global maxima of the population distribution varies. Such an example is provided by 
the symmetric, triple-well potential 
\begin{equation}
\label{potentielFlacaEnMedio}
-\mathcal{W}(x)=\tfrac{1}{200}x^2  (x-2)^4 (x+2)^4,
\end{equation}
which is of the ``wide-narrow-wide''  type. We numerically depict in Figure \ref{fig:pic-GfG} the ground state  associated to this fitness function, for different values of the mutation rate. As explained above, the ground state is  bi-modal for small $\sigma$ and uni-modal for large $\sigma$. More interestingly is that, for intermediate values of $\sigma$, the ground state is trimodal. Hence, the combination of the position of the wells of the potential and of the value of the parameter $\sigma$ is of great importance on the number of emerging phenotypes. 

\begin{figure}
\begin{center}
\includegraphics[width=12.5cm]{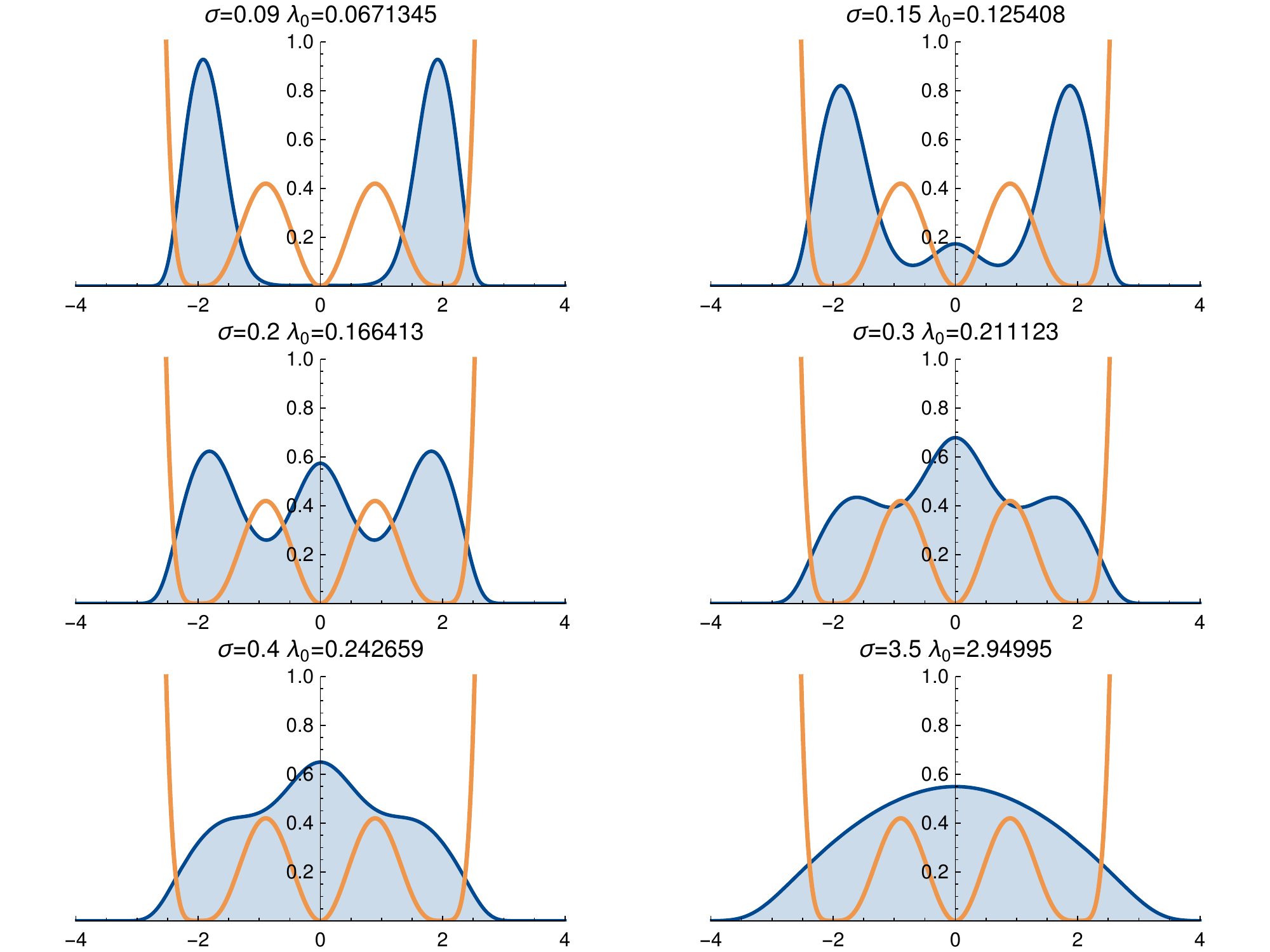}
\caption{Ground state $\phi_0(\sigma)$, increasing parameter $\sigma$, in the case 
of a ``wide-narrow-wide'' potential.}
\label{fig:pic-GfG}
\end{center}
\end{figure}

\section{Discussion}\label{s:discussion}

Our motivation is to understand the so-called branching phenomena, that is the splitting of a population structured by a phenotypic trait from uni-modal to multi-modal distribution.

We consider a population submitted to mutation and selection, thus standing in the framework of the dynamics of adaptation, see \cite{DL96}, \cite{D04}, \cite{DJMP05}, \cite{MPW12}, \cite{Cal-Cud-Des-Rao} among others. The retained model is the replicator-mutator equation, which is a deterministic integro-differential model \cite{PhysRevLett.76.4440}, \cite{Biktashev2014}, \cite{AlfaroCarles14, AlfaroCarles2017}, \cite{Gil-17}. The growth term involves a confining fitness function --- which prevents the possibility of ``escaping to infinity''--- to which the mean fitness is subtracted. Hence, if the initial data is a probability density then so is the solution for later times.

For this model, we have shown the following new mathematical results: the associated Cauchy problem is well-posed and the solution is written explicitly thanks to some underlying Schrödinger eigenelements. This requires the reduction to a linear equation via a change of unknown, the use of Lions' theorem and the derivation of rather non-standard estimates on the eigenelements. As a consequence of the expression of the solution, we deduce that the long time behaviour is determined by the principal eigenfunction or ground state.

Hence, the issue of branching reduces to the issue of the shape of the ground state. In a small mutation regime, we have presented sufficient conditions on the fitness function (symmetric, with two global maxima) for the population to split to bi-modality. Also, still in the small mutation and symmetric fitness regime, the widest global maxima of the fitness function are selected, thus revealing the number of emerging phenotypes. Last, we have underlined that  the number of maxima of the ground state and their value are determined by a combination of the fitness  function (symmetric or not, position of the wells) and the mutation parameter: the population density can be concentrated around some intervals of phenotypic trait in different proportions, corresponding to the emergence of well identified phenotypes.
 
The branching phenomena have recently received more attention \cite{WI13}, \cite{MeleardMirrahimi2015}, \cite{Ito-Sas-16}, \cite{Gil-17}, but to the best of our knowledge, this is the first work where it is obtained through the rather simple replicator-mutator equation \eqref{eq}. However, further investigations remain to be performed for a better understanding of the interplay between the fitness function and the mutation parameter, as sketched in Section \ref{s:branching}. Another relevant information for biological purposes would be an estimate of the time needed for a uni-modal population to branch.

\subsection*{Acknowledgements}
 The authors are  grateful to R\'emi Carles for suggesting Lemma \ref{lem:normes} and continuous encouragement, and to Bernard Helf{}fer for invaluable comments and essential remarks. They also thank Alexandre Eremenko and  Christian Remling for very valuable  discussions. Mario Veruete is grateful for the support of the National Council for Science and Technology of Mexico.


\bibliographystyle{siam}
\bibliography{mabiblio}

\begin{thebibliography}{10}

\bibitem{Agmon}
{\sc S.~Agmon}, {\em Lectures on Exponential Decay of Solutions of Second-Order
  Elliptic Equations: Bounds on Eigenfunctions of N-Body Schrodinger
  Operations. (MN-29)}, Princeton University Press, 1982.

\bibitem{AlfaroCarles14}
{\sc M.~Alfaro and R.~Carles}, {\em Explicit solutions for replicator-mutator
  equations: extinction versus acceleration}, SIAM J. Appl. Math., 74 (2014),
  pp.~1919--1934.

\bibitem{AlfaroCarles2017}
{\sc M.~Alfaro and R.~Carles}, {\em Replicator-mutator equations with quadratic
  fitness}, Proceedings of the American Mathematical Society, 145 (2017),
  pp.~5315--5327.

\bibitem{AllenRosenbloom2012}
{\sc B.~Allen and D.~Rosenbloom}, {\em Mutation rate evolution in replicator
  dynamics}, Bulletin of Mathematical Biology, 74 (2012), pp.~2650--2675.

\bibitem{Biktashev2014}
{\sc V.~N. Biktashev}, {\em A simple mathematical model of gradual {D}arwinian
  evolution: emergence of a {G}aussian trait distribution in adaptation along a
  fitness gradient}, J. Math. Biol., 68 (2014), pp.~1225--1248.

\bibitem{Brandon2013}
{\sc D.~Brandon and S.~Nasser}, {\em Exact and approximate solutions to
  {S}chr{\"o}dinger's equation with decatic potentials}, Central European
  Journal of Physics, 11 (2013), pp.~279--290.

\bibitem{Bra-Lie-76}
{\sc H.~J. Brascamp and E.~H. Lieb}, {\em On extensions of the
  {B}runn-{M}inkowski and {P}r\'ekopa-{L}eindler theorems, including
  inequalities for log-concave functions, and with an application to the
  diffusion equation}, J. Functional Analysis, 22 (1976), pp.~366--389.

\bibitem{Brezis}
{\sc H.~Brezis}, {\em Functional Analysis, Sobolev Spaces and Partial
  Differential Equations}, Springer-Verlag New York, 2011.

\bibitem{Cal-Cud-Des-Rao}
{\sc A.~Calsina, S.~Cuadrado, L.~Desvillettes, and G.~Raoul}, {\em Asymptotic
  profile in selection-mutation equations: {G}auss versus {C}auchy
  distributions}, J. Math. Anal. Appl., 444 (2016), pp.~1515--1541.

\bibitem{PhysRevA.43.3241}
{\sc R.~N. Chaudhuri and M.~Mondal}, {\em Improved hill determinant method:
  General approach to the solution of quantum anharmonic oscillators}, Phys.
  Rev. A, 43 (1991), pp.~3241--3246.

\bibitem{DL96}
{\sc U.~Dieckmann and R.~Law}, {\em The dynamical theory of coevolution: a
  derivation from stochastic ecological processes}, J. Math. Biol., 34 (1996),
  pp.~579--612.

\bibitem{D04}
{\sc O.~Diekmann}, {\em A beginner's guide to adaptive dynamics}, in
  Mathematical modelling of population dynamics, vol.~63 of Banach Center
  Publ., Polish Acad. Sci., Warsaw, 2004, pp.~47--86.

\bibitem{DJMP05}
{\sc O.~Diekmann, P.-E. Jabin, S.~Mischler, and B.~Perthame}, {\em The dynamics
  of adaptation: an illuminating example and a {H}amilton-{J}acobi approach},
  Theoretical Population Biology, 67 (2005), pp.~257--271.

\bibitem{Dji-Duc-Fab-17}
{\sc R.~Djidjou-Demasse, A.~Ducrot, and F.~Fabre}, {\em Steady state
  concentration for a phenotypic structured problem modeling the evolutionary
  epidemiology of spore producing pathogens}, Math. Models Methods Appl. Sci.,
  27 (2017), pp.~385--426.

\bibitem{Eremenko2008}
{\sc A.~Eremenko, A.~Gabrielov, and B.~Shapiro}, {\em High energy
  eigenfunctions of one-dimensional {S}chr{\"o}dinger operators with polynomial
  potentials}, Computational Methods and Function Theory, 8 (2008),
  pp.~513--529.

\bibitem{Eremenko2008bis}
\leavevmode\vrule height 2pt depth -1.6pt width 23pt, {\em Zeros of
  eigenfunctions of some anharmonic oscillators}, Annales de l’institut
  Fourier, 58 (2008), pp.~603--624.

\bibitem{Gagelman2012}
{\sc J.~Gagelman and H.~Yserentant}, {\em A spectral method for
  {S}chr{\"o}dinger equations with smooth confinement potentials}, Numerische
  Mathematik, 122 (2012), pp.~383--398.

\bibitem{Gil-17}
{\sc M.-E. Gil, F.~Hamel, G.~Martin, and L.~Roques}, {\em Mathematical
  properties of a class of integro-differential models from population
  genetics}, SIAM J. Appl. Math., 77 (2017), pp.~1536--1561.

\bibitem{book:832770}
{\sc P.~Haccou, P.~Jagers, and V.~Vatutin}, {\em Branching Processes:
  Variation, Growth, and Extinction of Populations}, Cambridge Studies in
  Adaptive Dynamics 5, Cambridge University Press, first edition~ed., 2008.

\bibitem{Hel-book-88}
{\sc B.~Helffer}, {\em Semi-classical analysis for the {S}chr\"odinger operator
  and applications}, vol.~1336 of Lecture Notes in Mathematics,
  Springer-Verlag, Berlin, 1988.

\bibitem{Hel-Rob-82}
{\sc B.~Helffer and D.~Robert}, {\em Asymptotique des niveaux d'\'energie pour
  des hamiltoniens \`a un degr\'e de libert\'e}, Duke Math. J., 49 (1982),
  pp.~853--868.

\bibitem{Hel-Sjo-85}
{\sc B.~Helffer and J.~Sj\"ostrand}, {\em Puits multiples en limite
  semi-classique. {II}. {I}nteraction mol\'eculaire. {S}ym\'etries.
  {P}erturbation}, Ann. Inst. H. Poincar\'e Phys. Th\'eor., 42 (1985),
  pp.~127--212.

\bibitem{Hel-Sjo-94}
{\sc B.~Helffer and J.~Sj\"ostrand}, {\em On the correlation for {K}ac-like
  models in the convex case}, J. Statist. Phys., 74 (1994), pp.~349--409.

\bibitem{Ito-Sas-16}
{\sc H.~Ito and A.~Sasaki}, {\em Evolutionary branching under multi-dimensional
  evolutionary constraints}, J. Theoret. Biol., 407 (2016), pp.~409--428.

\bibitem{PhysRevLett.80.2012}
{\sc D.~Kessler and H.~Levine}, {\em Mutator dynamics on a smooth evolutionary
  landscape}, Phys. Rev. Lett., 80 (1998), pp.~2012--2015.

\bibitem{book:1324333}
{\sc M.~Kimmel and D.~Axelrod}, {\em Branching Processes in Biology},
  Interdisciplinary Applied Mathematics 19, Springer-Verlag New York, 2~ed.,
  2015.

\bibitem{MeleardMirrahimi2015}
{\sc H.~Leman, S.~M\'el\'eard, and S.~Mirrahimi}, {\em Influence of a spatial
  structure on the long time behavior of a competitive {L}otka-{V}olterra type
  system}, Discrete Contin. Dyn. Syst. Ser. B, 20 (2015), pp.~469--493.

\bibitem{lions_problemes_1968}
{\sc J.~L. Lions and E.~Magenes}, {\em Problemes aux limites non homogenes et
  applications. Vol. 1. Vol. 1.}, Dunod, 1968.

\bibitem{Lor-Mir-Per-11}
{\sc A.~Lorz, S.~Mirrahimi, and B.~t. Perthame}, {\em Dirac mass dynamics in
  multidimensional nonlocal parabolic equations}, Comm. Partial Differential
  Equations, 36 (2011), pp.~1071--1098.

\bibitem{Martin1541}
{\sc G.~Martin and L.~Roques}, {\em The nonstationary dynamics of fitness
  distributions: Asexual model with epistasis and standing variation},
  Genetics, 204 (2016), pp.~1541--1558.

\bibitem{MPW12}
{\sc S.~Mirrahimi, B.~Perthame, and J.~Y. Wakano}, {\em Evolution of species
  trait through resource competition}, J. Math. Biol., 64 (2012),
  pp.~1189--1223.

\bibitem{rakotoson}
{\sc J.~E. Rakotoson and J.~M. Rakotoson}, {\em Analyse fonctionnelle
  appliquée aux équations aux dérivée partielles}, Presses Universitaires
  de France, 1999.

\bibitem{ReedSimonVol4}
{\sc M.~Reed and B.~Simon}, {\em Methods of Modern Mathematical Physics (vol
  IV): Analysis of Operators}, Academic Press, 1978.

\bibitem{RBW08}
{\sc I.~M. Rouzine, E.~Brunet, and C.~O. Wilke}, {\em {T}he traveling-wave
  approach to asexual evolution: Muller's ratchet and speed of adaptation},
  Theor. Popul. Biol., 73 (2008), pp.~24--46.

\bibitem{RWC02}
{\sc I.~M. Rouzine, J.~Wakekey, and J.~M. Coffin}, {\em {T}he solitary wave of
  asexual evolution}, Proc. Natl. Acad. USA, 100 (2003), pp.~587--592.

\bibitem{SimonAsymptoticEigenvalue}
{\sc B.~{Simon}}, {\em {Semiclassical analysis of low lying eigenvalues. I:
  Non-degenerate minima: Asymptotic expansions.}}, {Ann. Inst. Henri
  Poincar\'e, Nouv. S\'er., Sect. A}, 38 (1983), pp.~295--308.

\bibitem{SG10}
{\sc P.~D. Sniegowski and P.~J. Gerrish}, {\em {B}eneficial mutations and the
  dynamics of adaptation in asexual populations}, Phil. Trans. R. Soc. B, 365
  (2010), pp.~1255--1263.

\bibitem{Takhtajan}
{\sc L.~A. Takhtajan}, {\em Quantum Mechanics for Mathematicians}, Graduate
  Studies in Mathematics, American Mathematical Society, 2008.

\bibitem{titchmarsh_eigenfunction_1946}
{\sc E.~C. Titchmarsh}, {\em Eigenfunction expansions associated with second
  order differential equations}, Oxford At The Clarendon Press, 1946.

\bibitem{PhysRevLett.76.4440}
{\sc L.~Tsimring, H.~Levine, and D.~Kessler}, {\em Rna virus evolution via a
  fitness-space model}, Phys. Rev. Lett., 76 (1996), pp.~4440--4443.

\bibitem{Wak-Fun-Yok-17}
{\sc J.~Y. Wakano, T.~Funaki, and S.~Yokoyama}, {\em Derivation of
  replicator-mutator equations from a model in population genetics}, Jpn. J.
  Ind. Appl. Math., 34 (2017), pp.~473--488.

\bibitem{WI13}
{\sc J.~Y. Wakano and Y.~Iwasa}, {\em {E}volutionary {B}ranching in a {F}inite
  {P}opulation: {D}eterministic {B}ranching {\it vs.} {S}tochastic
  {B}ranching}, Genetics, 193 (2013), pp.~229--241.

\bibitem{XieWangFu}
{\sc Q.~Xie, L.~Wang, and J.~Fu}, {\em Analytical solutions for a class of
  double-well potentials}, Physica Scripta, 90 (2015), p.~045204.

\bibitem{Zaslavski}
{\sc O.~Zaslavskii and V.~UI'yanov}, Journal of Experimental and Theoretical
  Physics, 87 (1984), p.~991.

\end{thebibliography}

\end{document}